\documentclass[reqno,11pt]{amsart}
\usepackage{amssymb,amsmath,amsthm,amstext,amsfonts}
\usepackage{amsmath,amstext,amsthm,amsfonts}
\usepackage[dvips]{graphicx}
\usepackage[dvips]{graphicx}
\usepackage{psfrag}
\usepackage{color}

\makeatletter \@addtoreset{equation}{section} \makeatother

\renewcommand\thetable{\thesection.\@arabic\c@table}

\theoremstyle{plain}
\newtheorem{maintheorem}{Theorem}

\newtheorem{theorem}{Theorem }[section]
\newtheorem{proposition}[theorem]{Proposition}
\newtheorem{lemma}[theorem]{Lemma}
\newtheorem{corollary}[theorem]{Corollary}
\theoremstyle{definition} \theoremstyle{remark}
\newtheorem{remark}[theorem]{Remark}
\newtheorem{definition}[theorem]{Definition}
\newtheorem{conjecture}{Conjecture}
\newenvironment{pfof}[1]{\vspace{1ex}\noindent{\bf Proof of
#1}\hspace{0.5em}}{\hfill\qed\vspace{1ex}}

\newcommand{\R}{{\mathbb R}}

\newcommand{\C}{{\mathbb C}}
\newcommand{\Z}{{\mathbb Z}}

 \newcommand{\cF}{{\mathcal F}}
 \newcommand{\cW}{{\mathcal W}}
 \newcommand{\cH}{{\mathcal H}}
 \newcommand{\cU}{{\mathcal U}}

 \newcommand{\Sup}{{\sup}}

\newcommand{\Lip}{\operatorname{Lip}}
\newcommand{\supp}{\operatorname{supp}}
\newcommand{\Int}{\operatorname{Int}}
\newcommand{\diver}{\operatorname{div}}
\newcommand{\qand}{\quad\text{and}\quad}

\begin{document}

\title{Rapid mixing for the Lorenz attractor and \\
  statistical limit laws for their time-$1$ maps}

\author{V. Ara\'ujo and I. Melbourne and P. Varandas}

\address{Vitor Ara\'ujo and Paulo Varandas,
 Departamento de Matem\'atica, Universidade Federal da Bahia\\
Av. Ademar de Barros s/n, 40170-110 Salvador, Brazil.}
\email{vitor.d.araujo@ufba.br,
 www.sd.mat.ufba.br/$\sim$vitor.d.araujo}
\email{paulo.varandas@ufba.br, www.pgmat.ufba.br/varandas/}

\address{Ian Melbourne,
Institute of Mathematics, University of Warwick, Coventry CV4 7AL, UK}
\email{i.melbourne@warwick.ac.uk}

\thanks{I.M. was partially supported by a Santander Staff
  Mobility Award at the University of Surrey, by a European
  Advanced Grant StochExtHomog (ERC AdG 320977) and by CNPq
  (Brazil) through PVE grant number 313759/2014-6.  V.A. and
  P.V. were partially supported by CNPq,
  PRONEX-Dyn.Syst. and FAPESB (Brazil).  This research has
  been supported in part by EU Marie-Curie IRSES
  Brazilian-European partnership in Dynamical Systems
  (FP7-PEOPLE-2012-IRSES 318999 BREUDS).  We are grateful to
  Oliver Butterley for very helpful discussions regarding
  the regularity of the strong stable foliation and to the
  anonymous referees for pointing out several details in our
  argument that had to be addressed, greatly improving the
  final text of this work.
}

 \begin{abstract}
   We prove that every geometric Lorenz attractor satisfying
   a strong dissipativity condition has superpolynomial
   decay of correlations with respect to the unique SRB
   measure.  Moreover, we prove the Central Limit Theorem
   and Almost Sure Invariance Principle for the time-1 map
   of the flow of such attractors.  In particular, our
   results apply to the classical Lorenz attractor.
 \end{abstract}

 \date{20 November 2013; revised 3 September 2015}

\maketitle


 \section{Introduction}  \label{sec-intro}


The statistical point of view on Dynamical Systems is
one of the most useful tools available for the study of
the asymptotic behavior of transformations or
flows. Statistical properties are often easier to study than
pointwise behavior, since the future behavior of an
initial data point can be unpredictable, but
statistical properties are often regular and with
simpler description.  

One of the main concepts introduced is the
notion of \emph{physical} (or \emph{Sinai-Ruelle-Bowen}
(SRB)) measure for a flow (or transformation). An
invariant probability measure $\mu$ for a flow $Z_t$ is
a physical probability measure if the subset of points
$z$ satisfying for all continuous functions $w$
\begin{align*}
  \lim_{t\to+\infty}\frac1t\int_0^t w\big(Z_s(z)\big) \,
  ds = \int w\,d\mu,
\end{align*}
has positive volume in the ambient
space. These time averages are in principle physically
observable if the flow models a real world phenomenon
admitting some measurable features.

In 1963, the meteorologist Edward Lorenz published in the
Journal of Atmospheric Sciences \cite{Lo63} an example of a
polynomial system of differential equations
\begin{align}
  \label{e-Lorenz-system}
\dot x &= 10(y - x)
 \nonumber \\
\dot y &= 28x -y -xz
\\
\dot z &= xy - \textstyle{\frac83}z
 \nonumber
\end{align}
as a very simplified model for thermal fluid convection,
motivated by an attempt to understand the foundations of
weather forecast.

Numerical simulations performed by Lorenz for an open
neighborhood of the chosen parameters suggested that almost
all points in phase space tend to a {\em chaotic attractor},
whose well known picture can be easily found in the
literature.

The mathematical study of these equations began with
the geometric Lorenz flows, introduced independently by
Afra{\u\i}movi{\v{c}} {\em et al.}~\cite{Afr77} and
Guckenheimer \& Williams~\cite{GucWil79,Wil79} as
an abstraction of the numerically observed features of
solutions to (\ref{e-Lorenz-system}). The geometric
flows were shown to possess a ``strange'' attractor
with sensitive dependence on initial conditions.  It is
well known, see e.g. \cite{APPV}, that geometric Lorenz
attractors have a unique SRB (or physical) measure.
Tucker \cite{Tucker} showed that the attractor of the
classical Lorenz equations \eqref{e-Lorenz-system} is
in fact a geometric Lorenz attractor (see Remark~\ref{rmk:Tucker} below).  
For more on the
rich history of the study of this system of equations,
the reader can consult \cite{viana2000i, AraPac2010}.

An invariant probability measure $\mu$ for a flow is
mixing if
\begin{align*}
  \mu(Z_t(A)\cap B)\to\mu(A)\mu(B)
\end{align*}
as $t\to\infty$ 
for all measurable sets $A,B$.  Mixing
for the SRB measure of geometric Lorenz attractors was
proved in \cite{LMP05} and, by~\cite{Tucker}, this
includes the classical Lorenz attractor~\cite{Lo63}.

Results on the speed of convergence in the limit above,
that is, of rates of mixing for the Lorenz
attractor were obtained only recently: a first result
on robust exponential decay of correlations was proved
in \cite{ArVar} for a nonempty open subset of
geometric Lorenz attractors.  However, this open set
does not contain the classical Lorenz attractor.
Also, it follows straightforwardly from~\cite{Melb09} that a
$C^2$-open and $C^\infty$-dense set of geometric Lorenz
flows have superpolynomial decay of correlations (in
the sense of \cite{dolgopyat98}).  It is likely, but
unproven, that this open and dense set includes the
classical Lorenz attractor.  

\subsection{Statement of results}
\label{sec:statement-results}

In this paper, we introduce an additional open assumption, {\em strong dissipativity}, that is satisfied by the classical Lorenz attractor, under which we can prove superpolynomial decay of correlations.

We consider
$C^{\infty}$ vector fields $G$ on $\R^3$ possessing an
equilibrium $p$ which is \emph{Lorenz-like}:
the eigenvalues of $DG_p$ are real and satisfy
\begin{align}\label{eq:Lorenz-like-equil}
  \lambda_{ss} < \lambda_s < 0 < -\lambda_s < \lambda_u.
\end{align} 
We say that $G$ is {\em strongly dissipative} if 
  the divergence of the vector field $G$ is strictly
  negative: there exists a constant $\delta>0$ such
  that $(\diver G)(x)\le-\delta$ for all $x\in U$,
and moreover the eigenvalues of the singularity at $p$ satisfy the 
additional constraint $\lambda_u+\lambda_{ss}<\lambda_s$. 
For the classical Lorenz equations~\eqref{e-Lorenz-system}, we have
\[
\diver G\equiv -\textstyle{\frac{41}{3}}, \quad
\lambda_s=-\textstyle{\frac83}, \quad \lambda_u\approx 11.83, \quad \lambda_{ss}\approx -22.83, 
\]
so the conditions~\eqref{eq:Lorenz-like-equil} and strong dissipativity are 
satisfied.

Let $\mathcal{U}$
denote the open set of $C^\infty$ vector fields having
a strongly dissipative geometric Lorenz attractor $\Lambda$; see
Section~\ref{sec:geometr-aspects-lore} for precise
definitions.  Given $G\in\mathcal{U}$, let $Z_t$ denote
the flow generated by $G$ and let $\mu$ denote the
unique SRB measure supported on $\Lambda$.

\begin{maintheorem}\label{thm:rapid}
Let $G\in\mathcal{U}$.
Then for all $\gamma>0$, there exists $C>0$ and $k\ge1$ such that for all
$C^k$ observables $v,w: \R^3\to
  \R$ and all $t>0$,
\[
\Big|
\int v \; w\circ Z_t \, d\mu -
\int v\,d\mu  \int w\, d\mu\Big|
\le 
C\|v\|_{C^k} \|w\|_{C^k}t^{-\gamma}.
\]
\end{maintheorem}

By \cite{HoMel}, geometric Lorenz flows satisfy the
Central Limit Theorem (CLT) for H\"older observables.
A stronger property is the CLT for the time-$1$ map
$Z=Z_1$ which is only partially hyperbolic.  By
Theorem~\ref{thm:rapid}, $Z$ has superpolynomial decay
of correlations.  Following~\cite{MelTor02}, we use
this information to prove the CLT for time-$1$ maps of
geometric Lorenz flows thereby verifying Conjecture 4
in~\cite{ArVar}.

\begin{maintheorem}\label{thm:CLT}
  Let $G\in\mathcal U$.  Then there exists
  $k\ge1$ such that for all $C^k$ observables $v:
  \R^3 \to \R$ there exists $\sigma\ge0$ such that
\begin{equation*}
  \frac{1}{\sqrt{n}} 
  \Bigg[\sum_{j=0}^{n-1} v\circ Z^j \,
  -\,  n\int v \;d\mu\Bigg] \xrightarrow{\mathcal D}
  \mathcal N(0,\sigma^2)
\end{equation*}
where the convergence is in distribution.

Moreover, if $\sigma^2=0$, then
for every periodic point $q\in\Lambda$, there exists $T>0$ (independent of $v$) such that
$\int_0^T v(Z_tq)\,dt=0$. 
\end{maintheorem}

\begin{remark} \label{rmk-sigma}
Since there are infinitely many distinct periodic solutions in $\Lambda$,
it follows from the final statement of Theorem~\ref{thm:CLT} that 
the family of $C^k$ observables $v:\R^3\to\R$ for which
  $\sigma^2=0$ forms an infinite codimension family in the
  space of all $C^k$ observables.
\end{remark}

By~\cite{HoMel}, geometric Lorenz flows satisfy also an Almost Sure Invariance Principle (ASIP) for vector-valued observables $v:\R^3\to\R^d$.
Such a result is currently unavailable for the time-$1$ map $Z$, but we are able to prove a scalar ASIP.

\begin{maintheorem}\label{thm:ASIP}
  Let $G\in\mathcal U$.  There exists $k\ge1$ such that for
  all $C^k$ observables $v: \R^3 \to \R$ the ASIP holds for
  the time-$1$ map: passing to an enriched probability
  space, there exists a sequence $X_0,X_1,\ldots$ of iid
  normal random variables with mean zero and variance
  $\sigma^2$ (as in Theorem~\ref{thm:CLT}), such that
\[
\sum_{j=0}^{n-1}v\circ Z^j=n\int v\,d\mu+\sum_{j=0}^{n-1}
X_j+O(n^{1/4}(\log n)^{1/2}(\log\log n)^{1/4}),\; a.e.
\]
\end{maintheorem}

\begin{remark} The ASIP implies the CLT and also the
  functional CLT (weak invariance principle), and the law of
  the iterated logarithm together with its functional
  version, as well as numerous other results.
  See~\cite{PhilippStout75} for a comprehensive list.
\end{remark}

\subsection{Comments and organization of the paper}
\label{sec:comments-organiz-pap}

In Section~\ref{sec:geometr-aspects-lore}, we recall basic
properties of geometric Lorenz attractors.  In
Section~\ref{sec:tempor-distort-funct}, we define the
temporal distortion function and prove a result about the
dimension of its range. This is the main new ingredient in
the proof of Theorem~\ref{thm:rapid} in
Section~\ref{sec:fast-mixing-decay}.

In Section~\ref{sec-ASIPsemiflow}, we prove a general result
on the ASIP for time-$1$ maps of nonuniformly expanding
semiflows.   In Section~\ref{sec-sigma}, we prove that the ASIP is
typically nondegenerate.
In Section~\ref{sec-ASIPflow}, we prove 
Theorems~\ref{thm:CLT} and~\ref{thm:ASIP}.

It is natural to extend all these results to more general
singular-hyperbolic attractors (formerly referred to as
Lorenz-like flows), that is, transitive attracting sets of
three-dimensional flows having finitely many Lorenz-like
singularities and a volume hyperbolic structure; see
e.g.~\cite{AraPac2010} for the precise definitions. Indeed,
analogously to the geometric Lorenz case, it is possible to
reduce the dynamics of these attractors to a piecewise
expanding $C^{1+\epsilon}$ one-dimensional map; see
e.g.~\cite[Chapter 6]{AraPac2010} or~\cite[Section
4]{ArGalPac} for a detailed presentation.

\begin{conjecture}
  \label{conj:singhyp}
  Let $\cU$ denote the open set of $C^\infty$ vector fields
  having a singular-hyperbolic attractor on a given compact
  three-dimensional manifold. Then the results stated in
  Theorems~\ref{thm:rapid}, \ref{thm:CLT} and \ref{thm:ASIP}
  are true for all $G\in\cU$.
\end{conjecture}

There exists a natural generalization of
singular-hyperbolicity for higher-dimensional attractors,
known as sectional-hyperbolicity; see e.g.~\cite[Sections
5.2 \& 8.2]{AraPac2010} and also \cite{MeMor06}. In this
setting both the stable and the unstable manifolds of points
in the attractor need not be codimension one embedded
submanifolds, which makes analysis of these singular flows
challenging.

\begin{conjecture}
  \label{conj:highdim}
    Let $\cU$ denote the open set of $C^\infty$ vector fields
  having a sectional-hyperbolic attractor in a given compact
  finite dimensional manifold. Then the results stated in
  Theorems~\ref{thm:rapid}, \ref{thm:CLT} and \ref{thm:ASIP}
  are true for all $G\in\cU$. 
\end{conjecture}

\subsection*{Notation} Throughout, $C$ is used to
denote a constant whose value may change from line to line.


\section{Geometric aspects of Lorenz attractors}
\label{sec:geometr-aspects-lore}

\subsection{Geometric Lorenz attractors}
\label{sec:geomL}

We define here the open set $\cU$ of $C^{\infty}$
vector fields exhibiting strongly dissipative geometric Lorenz attractors
and we describe the basic structure of such attractors;
see e.g.~\cite{AraPac2010}.

Let $G$ be a strongly dissipative $C^{\infty}$ vector field on $\R^3$ 
possessing a Lorenz-like equilibrium, which we suppose without loss to be at $0$.
We assume that the flow $Z_t$ is $C^{1+\epsilon}$ linearizable in a neighborhood
of $0$ which, by a suitable choice of coordinates, can be
assumed to contain the cube $[-1,1]^3$. 
(It follows from~\cite[Theorem~12.1]{Hartman02} that smooth linearizability holds for an open and dense set of such vector fields.)
Choose coordinates $x_1,x_2,x_3$ corresponding to the eigenspaces
of $\lambda_u$, $\lambda_{ss}$, $\lambda_s$ respectively.
We define the
cross-section $X = \{(x_1 , x_2 , 1) : |x_1 |, |x_2
|\le1\}$ and the
Poincar\'e map $f : X\to X$.
For $x\in X$ we write
$f(x)=Z_{r(x)}(x)$ where $r:X\to\R^+$ is the Poincar\'e first
return time to $X$, also referred to as the roof function.

We assume that there exists  a global exponentially contracting $f$-invariant \emph{stable
  foliation}. That is, there is a compact neighborhood
  $N\subset X$ of $(0,0,1)$ satisfying $f (N \setminus \{x_1 =
0\})\subset X$ and a 
partition $\cW^s_f$ of $N$  consisting of $C^\infty$ one-dimensional disks called {\em stable leaves} (including the ``singular leaf'' $\{x_1=0\}$).
Let $W^s_f(x)$ denote the stable leaf containing $x$.  Then it is required that
$f(W^s_f(x))\subset W^s_f(f(x))$ for all $x\in N$ and that there exist
constants $C>0$, $\lambda_0\in(0,1)$ such that
$|f^n(x)-f^n(x')|\le C\lambda_0^n$ for all 
$x,x'$ in the same leaf and all $n\ge1$.

Moreover, we assume that $\cW_f^s$ is a $C^{1+\epsilon}$ foliation,
meaning that $N$ can be chosen so that 
there is a $C^{1+\epsilon}$ change of coordinates from the interior of
$N$ onto $(-1,1)\times(-1,1)$ transforming stable leaves into vertical lines.

Shrinking $N$ if necessary, we can arrange that each stable leaf
intersects $\bar X = \{(x_1 , 0, 1) : |x_1|\le 1\} \cong
[-1, 1]$ in a single point. Define the $C^{1+\epsilon}$
projection $\pi:X\to \bar
X$ given by holonomy along the stable leaves 
(so $\pi(x)=W_f^s(x)\cap\bar X$).
Quotienting along stable leaves,
we obtain a $C^{1+\epsilon}$ one-dimensional map $\bar
f : \bar X \to \bar X$ with a singularity at $0$: $\bar
f(x_1)=\pi(f(x_1,0,1))$.

\begin{lemma}[Proposition 2.6 in
  \cite{HoMel}]\label{le:prop2.6}
   Let $\eta=-\lambda_s/\lambda_u\in(0,1)$.
  \begin{enumerate}
  \item $\bar f'$ is H\"older on $\bar X\setminus\{0\}$: $\bar f'(x)=|x|^{\eta-1}g(x)$ with $g\in
    C^{\eta\epsilon}(\bar X), g>0$;
  \item the roof function has a logarithmic singularity
    at $0$: $r=h_1+h_2$ with
    $h_1(x)=-\lambda_u^{-1}\log|\pi(x)|$ and $h_2\in
    C^{\epsilon}(X)$.
  \end{enumerate}
\end{lemma}

In addition, we assume that $ \bar f$ is \emph{uniformly
  expanding}: there are constants $\lambda_1 > 1$ and
$c > 0$ such that $|( \bar f^n)'(x)|\ge c \lambda_1^n$ for
all $x\in \bar X$ and $n>1$.

As in~\cite{LMP05}, we assume further that $\bar f$ is
\emph{locally eventually onto (l.e.o.)}; namely that for any
open set $U\subset\bar X\setminus\{0\}$, there exists
$k\ge0$ such that $f^kU$ contains $(0,1)$.   (More
generally, it suffices that almost every point in $\bar X$
has dense preimages in $\bar X$.  However, the l.e.o.\
property is standard in the literature and holds for the
classical Lorenz attractor~\cite{Tucker}.)

%

Considering $U=\bigcup_{x\in X} Z_{[0,r(x)]}(x)$ we obtain
a closed neighborhood of $[-1,1]^3$ and, in what
follows, we denote by $\Lambda=\bigcap_{t>0}{Z_t(U)}$ the
\emph{geometric Lorenz attractor} of the vector field
$G$.  It can be shown that $\Lambda$ is
compact, volume hyperbolic, has a dense regular orbit
and has zero volume (Lebesgue measure in $\R^3$); see
e.g. \cite{AraPac2010,AAPP}.

\subsection{Volume hyperbolicity, dissipativity and
  consequences}
\label{sec:volume-hyperb-conseq}
We recall that, in our three-dimensional setting, volume
hyperbolicity means that there exists a $DZ_t$-invariant
singular-hyperbolic splitting of the tangent bundle over
$\Lambda$.   That is, there is a vector bundle splitting
$T_\Lambda\R^3=E\oplus F$ with $\dim E=1$, $\dim F=2$,
and there are constants $c>0$,
$\lambda\in(0,1)$, such that for all $x\in\Lambda$, $t>0$,
\begin{itemize}
\item \emph{the splitting is dominated}: $\|DZ_t \mid E_x\| \cdot 
\|DZ_{-t} \mid F_{X_t(x)}\| < c \, \lambda^t;$
\item \emph{$E$ is uniformly contracting}: $\|DZ_t \mid
  E_x\| < c \, \lambda^t;$
\item \emph{the area along $F$ is uniformly expanded}: $|\det
  DZ_t\mid F_x|\geq c \lambda^{-t}$.
\end{itemize}
The existence of the stable foliation $\cW^s_f$ of any small
cross-section to the flow of $G$ (such as $X$) is a
consequence of volume hyperbolicity for three-dimensional
smooth flows;  see e.g. \cite[Chapter 3, Section 3]{AraPac2010}.

An important consequence of domination, uniform contraction
along the stable direction $E$ and strong dissipativity for the
attractor $\Lambda$ is the existence of a $C^{1+\epsilon}$
global exponentially contracting $Z_t$-invariant
foliation $\cF^{ss}$, defined in a
neighborhood (which we may take to be $U$) of $\Lambda$.

\begin{lemma} \label{lem-foliate}
The strong stable foliation $\cF^{ss}$ is $C^{1+\epsilon}$ for some
$\epsilon>0$.
\end{lemma}

\begin{proof}
We apply
\cite[Theorem 6.2]{HP70} adapted to our
setting, since only domination and uniform contraction is
used in its proof. Indeed, a
sufficient condition to obtain $C^{1+\epsilon}$ regularity for the
strong stable foliation is that for some $t>0$,
\begin{align}\label{eq:alpha-smooth}
  \|DZ_t\mid E_x\|\cdot \|DZ_t\mid F_x\|^{1+\epsilon}\cdot \|DZ_{-t}\mid
  F_{Z_t(x)}\| <1
\end{align}
for all $x\in\Lambda$. 
(We note that the statement in~\cite{HP70} covers only the case $\epsilon=0$,
but it is standard that their result extends to the case $\epsilon>0$.)

For each $t\in\R$ we define $\eta_t:\Lambda\to\R$,
\[
\eta_t(x)=\log\Big\{\|DZ_t|E_x\|\cdot\|DZ_t|F_x\|^{1+\epsilon}
\|DZ_{-t}|F_{Z_tx}\|\Bigr\}.
\]
Note that $\{\eta_t,\;t\in\R\}$ is a continuous family of
continuous functions each of which is subadditive, that is,
$\eta_{s+t}(x)\le \eta_s(x)+\eta_t(Z_s(x))$.

Let $\mathcal{M}$ denote the set of flow-invariant ergodic probability measures on $\Lambda$.  We claim that for $\epsilon>0$ sufficiently small, and for each $m\in\mathcal{M}$, the limit
$\lim_{t\to\infty}\frac1t \eta(x)$ exists and is negative for $m$-almost every $x\in\Lambda$.
It then follows from~\cite[Proposition~3.4]{arbieto2010}
that there exists constants $C,\beta>0$ such that $\exp\eta_t(x)\le Ce^{-\beta t}$ for all $t>0$, $x\in\Lambda$.
In particular, for $t$ sufficiently large,
$\exp\eta_t(x)<1$ for all $x\in\Lambda$.  Hence condition~\eqref{eq:alpha-smooth} is satisfied for such $\epsilon$
and $t$ and the result follows.

It remains to verify the claim.
Let $m_0$ denote the Dirac delta concentrated at
$0$ and let $\mathcal{M}_1=\mathcal{M}\setminus\{m_0\}$.
We deal with the cases $m\in\mathcal{M}_1$ and $m=m_0$ separately.

Each $m\in\mathcal{M}_1$ has a zero Lyapunov exponent in the flow direction and two further Lyapunov exponents $\lambda_E(m)<0$ and $\lambda_F(m)>0$ associated
with the vector bundles $E$ and $F$ respectively.
Fix $m\in\mathcal{M}_1$.
For $m$-a.e. $x\in\Lambda$ we have
\begin{align} 
\label{eq-det}
 & \lim_{t\to\infty}\frac{1}{t} \log |\det Z_t(x)|=\lambda_E(m)+\lambda_F(m), \\
\label{eq-E} & \lim_{t\to\infty}\frac{1}{t} \log \|DZ_t|E_x\|=\lambda_E(m), \\
\label{eq-F}
& \lim_{t\to\infty}\frac{1}{t} \log \|DZ_t|F_x\|=\lambda_F(m), \\
\label{eq-Finv}
& \lim_{t\to\infty}\frac{1}{t} \log \|DZ_{-t}|F_{Z_tx}\|=0.
\end{align}
On the other hand, it follows from dissipativity that
 $\limsup_{t\to\infty}\frac{1}{t} \log |\det Z_t(x)|\le-\delta$ for all
$x$.  Hence we deduce from~\eqref{eq-det}
that $\lambda_E(m)+\lambda_F(m)\le -\delta$.
Moreover, $\lambda_F(m)\le \sup_\Lambda\|DG\|$, so
for $\epsilon>0$ sufficiently small (uniformly in $m$)
$\lambda_E(m)+(1+\epsilon)\lambda_F(m)<0$.
Using~\eqref{eq-E},~\eqref{eq-F}
and~\eqref{eq-Finv} together with the definition of
$\eta_t$, it follows that
$\lim_{t\to\infty}\frac1t \eta_t(x)<0$ for $m$-almost every $x\in\Lambda$.

It remains to consider the Dirac measure $m_0$.
By strong dissipativity, for $\epsilon$ sufficiently small,
$\frac1t\eta_t(0)=
\lambda_{ss}+(1+\epsilon)\lambda_u-\lambda_s<0$ for all $t$
as required.
\end{proof}

By Lemma~\ref{lem-foliate},
we may consider the
cross-section $X=\bigcup\{F^{ss}(x):x\in \bar X\}$ in the
place of the original cross-section $X$.   
Since the strong stable foliation is $C^{1+\epsilon}$ and the cross-section $X$ is
foliated by stable leaves over the smooth disk $\overline X$, it follows that
$X$ is a $C^{1+\epsilon}$ embedded surface in $\R^3$.
All the properties described so far are retained, with
the useful advantage that $W^s_f(x)=F^{ss}(x)$ for all $x\in X$ and
\begin{description}
\item[(C)] the first return time 
  $r:X\to\R^+$ of any given point in $X\setminus\{x_1=0\}$ (where
  $\{x_1=0\}$ now represents the leaf of $\cW^s_f$ through the point
  $0\in\bar X$) to $X$ is constant on the leaves of $\cW^s_f$,
  that is, $r(x)=r(\pi(x))$ for all $x\in X\setminus
  \{x_1=0\}$. Since the cross-section $X$ is a $C^{1+\epsilon}$
  embedded surface in $\R^3$, the roof function $r:X\to\R^+$ retains the properties mentioned in Lemma~\ref{le:prop2.6}(2); in particular $r$ is a
  $C^{1+\epsilon}$ function with a logarithmic singularity
  at $\{x_1=0\}$.
\end{description}
We keep the notation $\pi:X\to \bar X$ for the holonomy
along the leaves of $\cW^s_f$ to $\bar X$ and also $\bar f$ for
the one-dimensional $C^{1+\epsilon}$ quotient map of
$f:X\setminus\{x_1=0\}\to X$ over $\cW^s_f$.

Another consequence of volume hyperbolicity is that there
exists a field of cones $\tilde C_b(x)=\{(u,v)\in E_x\times
F_x: b\|v\|\ge\|u\|\}$ having width $b>0$ containing the $F$
subbundle over $\Lambda$ which admit a continuous
$DZ_t$-invariant extension $\hat C_b(x)$ to a neighborhood
of $\Lambda$. For the geometric Lorenz flow we can assume
without loss that this neighborhood coincides with $U$.

The invariance means that $DZ_t\cdot \hat C_b(x)\subset \hat
C_b(Z_t(x))$ for $x$ in an open neighborhood $U$ of
$\Lambda$ and $t>0$, where $b>0$ is small enough. Then the cones
$C_b(x)=\hat C_b(x)\cap T_xX$ on $T_xX$ are also
$D\bar f$-invariant and defined on the whole of $X\cap U$.

We say that a $C^1$ curve $\gamma$ in $X$ is a
\emph{$u$-curve} if $\gamma'(s)\subset C_b(\gamma(s))$ for
all parameter values $s$. The $D\bar f$-invariance of the field
of cones $C_b$ ensures that the image by $f$ of every
$u$-curve is sent into another $u$-curve. Moreover, the
tangent direction to the stable leaves $T_x\cW^s_f(x)$ is not
contained in the $C_b(x)$ cone and makes an angle bounded
away from zero with any vector inside $C_b(x)$, for all
$x\in X$, by the volume hyperbolicity assumption; see
\cite{APPV,AraPac2010}.

\begin{remark} \label{rmk:Tucker}
Tucker~\cite{Tucker} showed that the classical Lorenz equations have a robust nontrivial attractor $\Lambda$ containing the equilibrium at the origin.  It follows from
Morales~{\em et al.}~\cite{MPP04} that $\Lambda$ is a singular hyperbolic attractor that (in their words) ``resembles a geometric Lorenz attractor''.  In particular, it is immediate that all
of the properties listed above are satisfied except possibly for (i) strong dissipativity, (ii) the l.e.o property, and (iii) smoothness (class $C^{1+\epsilon}$) of
the contracting foliations $\cF^{ss}$ and $\cW_f^s$ for the flow and Poincar\'e map respectively.
We note that property~(i) is immediate for the classical Lorenz equations and condition~(ii) was verified in~\cite{Tucker}.
Regarding~(iii), it is claimed in~\cite[Section~2.4]{Tucker} that $\cW_f^s$ is a
smooth foliation but no details are provided.

Smoothness of the contracting foliations $\cF^{ss}$ and $\cW_f^s$ is not part of the definition of singular hyperbolic attractor, and hence is not discussed
in~\cite{MPP04}.  However, proofs of existence of an SRB measure with good statistical properties rely heavily on the smoothness of $\cW_f^s$.  Although this foliation is of codimension one, the fact that the Poincar\'e map $f$ is singular means that an extra argument is required; see for example
Robinson~\cite{Robinson92} and Rychlik~\cite[Section~4]{Rychlik90}.
In particular, our results apply to the open set of flows considered by~\cite{Robinson92,Rychlik90}, but these do not include the classical Lorenz equations.  However, as shown above in Lemma~\ref{lem-foliate}, the properties established by~\cite{MPP04,Tucker} combined with strong dissipativity guarantee smoothness of $\cF^{ss}$, and hence of $\cW_f^s$, for the classical parameters and nearby parameters.
\end{remark}

\subsection{Inducing and quotienting}
\label{sec:lorenz-attractors-as}

The geometric Lorenz attractor can be written as a
suspension flow $S_t:X^r\to X^r$ given by $S_t(x,s)=(x,s+t)$ on the space
$$
X^r=\{(x,t) \in X\times \R : 0\le t \le r(x)\} /\sim
$$
where $(x,r(x))\sim(f(x),0)$. Indeed, we can take the
conjugacy as $\Phi:X^r\to U, (x,s)\mapsto Z_s(x)$
(which is smooth) and the roof function $r:X\to\R^+$
has a logarithmic singularity at all points of
$X\cap\{x_1=0\}$, is smooth elsewhere and $f$ is a
non-uniformly hyperbolic map with invariant
stable foliation $\cW^s_f$. We denote also by $0$ the point
$\pi(\{x_1=0\})$ in what follows and since
$r\circ\pi=r$ we also write $r$ for the restriction $r:
\bar X\to\R^+$.

In particular, since $\bar f : \bar X\to \bar X$ is a
$C^{1+\epsilon}$-piecewise nonuniformly expanding map, there
exists a subset $\bar Y\subset \bar X$ and an inducing time
$\tau: \bar Y \to \Z^+$ such that $\bar F =\bar f^{\tau} :
\bar Y \to \bar Y$ is a $C^{1+\epsilon}$-piecewise expanding
Markov map with partition $\alpha_0$; see e.g.~\cite[Theorem
4.3]{ArVar}.  

Since $\bar f$ is l.e.o., for any specified point $x\in\bar X$ we can choose
$\bar Y$ to be an arbitrarily small open interval containing $x$;
see e.g.\ \cite{ALP}, where an
inductive construction procedure is described showing that
we can build a full branch Markov map $\bar F:\bar Y\to\bar Y$ as long as $\bar Y$ is a
neighborhood of a point with dense preimages. 

During the paper, we will consider various inducing schemes.  All of these
are full branch on an interval except for the one constructed in
Section~\ref{sec-double} which is the combination of two such inducing schemes.

Let $\alpha_0^n=\bigvee_{i=0}^{n-1}(\bar
F^i)^{-1}(\alpha_0)$ denote the $n$th refinement of $\alpha_0$, and set
  $\tau_n(y)=\sum_{j=0}^{n-1}\tau( \bar
  f^j(y))$, so that $\bar F^n(y)=\bar f^{\tau_n(y)}(y)$.

The most important features of $\bar f$ are the
following backward contraction and bounded distortion properties.
There exist constants $c_0>0, \lambda\in(0,1)$
such that for each $n\ge1$
\begin{description}
\item[backward contraction] 
  $\bar F^n\mid_{\alpha_0^n(y)} : \alpha_0^n(y)\to\bar Y$ is a
  $C^{1+\epsilon}$ diffeomorphism and
if $y'\in\alpha_0^n(y)$, then
\begin{align}\label{eq:backcontraction}
    | \bar f^i(y')-  \bar f^i(y)| \le
    c_0\lambda^{\tau_n(y)-i}|\bar F^n(y')-\bar F^n(y)|, \quad
    i=0,\dots,\tau_n(y)-1,
  \end{align}
  Moreover there is \emph{slow recurrence to the
    singular point}
  \begin{align}\label{eq:slow-recurrence}
    |\bar f^i(y)|\ge \sqrt\lambda^{\tau_n(y)-i}, \quad
    i=0,\dots,\tau_n(y)-1;
  \end{align}

\item[bounded distortion] if $y'\in\alpha_0^n(y)$, then
\begin{align}\label{eq:bdd-dist}
    \left| \frac{D\bar F^n(y)}{D\bar F^{n}(y')} -1 \right|
    \le c_0 |\bar F^{n}(y)-\bar F^{n}(y')|.
  \end{align}
\end{description}
We note that the induced map can be obtained by the methods presented in \cite{ALP}
and conditions \eqref{eq:backcontraction} and \eqref{eq:slow-recurrence}
follow from the definition of hyperbolic times (c.f. Definition~10 in \cite{ALP} with $b=1/2$).

Next we construct a piecewise uniformly hyperbolic
map $F:Y\to Y$ with infinitely many branches, which covers $\bar
F$, as follows:  Define
$Y=\bigcup\{W^s_f(y):y\in\bar Y\}$ to be the union
  of the stable leaves through $\bar Y$ and define
the Poincar\'e return map $F(y)=f^{\tau(\pi(y))}(y)$ for $y\in Y$.
We let $\alpha$ denote the
measurable partition of $Y$ whose elements are $\bigcup\{
W^s_f(x) : x\in a \}$ with $a\in\alpha_0$.
Also, we extend $\tau:\bar Y\to\Z^+$ to a function on $Y$ by
setting $\tau(y)=\tau(\pi y)$.

Let $\mu_{\bar Y}$ be the unique $\bar F$-invariant
absolutely continuous probability measure on $\bar Y$. It is
well-known that $r \in L^1(\mu_{\bar Y})$. It is then
standard that there exist unique invariant measures $\mu_Y$
for $F:Y\to Y$, $\mu_X$ for $f:X\to X$ and $\mu_{\bar X}$
for $\bar f:\bar X\to\bar X$ satisfying
$\pi_*(\mu_Y)=\mu_{\bar Y}$, $\pi_*(\mu_X)=\mu_{\bar X}$ and
also $\mu_X=\sum_{n\ge1} \sum_{j=0}^{n-1}
f^j_*(\mu_Y|\{\tau\circ\pi=n\})$ and $\mu_{\bar
  X}=\sum_{n\ge1} \sum_{j=0}^{n-1} \bar f^j_*(\mu_{\bar
  Y}|\{\tau=n\})$.  We have $\mu_Y\ll\mu_X$ and
$\mu_Y(Y)=1$, hence $\mu_X(Y)>0$; see
e.g.~\cite[Section~3]{ArVar} for more details.

\subsection{Local product structure}
\label{sec:local-product-struct}

Here we obtain an almost everywhere defined local product
structure for the induced hyperbolic map $F:Y\to Y$.
We have already seen that the Poincar\'e map $f:X\to X$ has a stable foliation
$\cW^s_f$ with leaves that cross $X$, and hence the induced map $F:Y\to Y$ has stable manifolds $W^s_F(y)=W^s_f(y)$ that cross $Y$.
In the next proposition, we construct local unstable manifolds for $F$ of uniform size, defined almost everywhere.

\begin{proposition}\label{pr:Wu_crosses}
For $\mu_Y$-almost every $y\in Y$, there exists a local
unstable manifold $W_F^u(y)\subset W_{loc,f}^u(y)$ that
crosses $Y$.
In particular, $\pi(W_F^u(y))=\bar Y$ for
$\mu_Y$-almost every $y\in Y$.
\end{proposition}

\begin{proof}
  We begin with the local unstable manifolds $W^u_{loc}$ for
  the flow $Z_t$.  It follows from Pesin theory (see
  e.g. \cite{BarPes2007}) that almost every point $p$ of
  $\Lambda$ with respect to the SRB measure $\mu$ admits a
  local unstable manifold $W^u_{loc}(p)$ which is a
  $C^{1+\epsilon}$-curve containing $p$ in its interior.
  By definition,
  $p'\in W^u_{loc}(p)$ if and only if $|Z_t(p')-Z_t(p)|\le
  C_p \lambda^{-t}$ for all $t<0$ (recall that
  $\lambda\in(0,1)$ and also that the constant $C_p$ depends
  on the leaf); see e.g.  \cite{APPV}. Since $\Lambda$ is an
  attractor, unstable leaves are contained in $\Lambda$.

  The smooth conjugacy $\Phi^{-1}$ sends these leaves into
  unstable leaves for the suspension flow $S_t$, which can
  be written locally as $\tilde
  W^u_{loc}(\Phi^{-1}(p))=\{(\gamma(s),t(s)): s\in[-1,1]\}$,
  where $\gamma:[-1,1]\to X$ and $t:[-1,1]\to\R$ are
  $C^{1+\epsilon}$ diffeomorphisms into their images. By its
  definition, the curve $\gamma$ is the local unstable
  manifold $W^u_{loc,f}(x)$ through $x=\gamma(0)$ with
  respect to $f$, that is, $x'\in W^u_{loc,f}(x)$ if and
  only if $|f^n(x')-f^n(x)|\le C_x' \lambda^{-n}$, for all
  $n<0$. We observe that the inverse images of $x$ and $x'$
  are all well defined since these points belong to the
  attractor $\Phi^{-1}(\Lambda)$ which is
  $S_t$-invariant. Moreover, the curve $\gamma$ is a graph
  $\gamma(s)=(x(s),y(x(s)))$ and $\gamma'$ is contained in a
  cone $\{(x',y')\in\R^2:|y'|<\xi|x'|\}$ for some $0<\xi<1$
  by the domination condition on $f$, consequence of the
  existence of dominating splitting for the flow on the
  attractor.

  We remark that $W^u_{loc,f}(x)$ is a $u$-curve which
  coincides with $W^u_F(x)$ if $x$ also belongs to
  $ Y$, and so the statements above hold for
  $\mu_Y$ almost every point $y\in Y$. Indeed, on
  the one hand, $W^u_{loc,f}(x)$ is formed by points whose
  preorbit is asymptotic to the preorbit of $x$, hence these
  preorbits contain the preorbits with respect to $F$, and
  so $W^u_{loc,f}(x)\subset W^u_F(x)$. On the other hand,
  this inclusion shows that $W^u_{loc,f}(x)$ and $W^u_F(x)$
  coincide in a neighborhood of $x$ inside $W^u_F(x)$. Since
  these unstable manifolds are contained in the attractor
  then, repeating the argument around each point $z\in
  W^u_{loc,f}(x)$, we see that the two manifolds coincide.

  In addition, the stable leaves through points $y\in
  W^u_F(x)$ are transverse to $W^u_F(y)$ and the angle
  between $T_y W^s_F(y)$ and $T_yW^u_F(y)$ is bounded away
  from zero. Hence 
  $\pi(W^u_F(y))$ is a neighborhood of $y_0=\pi(y)$ in $\bar Y$
   for $\mu_Y$-a.e. $y$.  

  Let $\alpha_0^n(y_0)$ denote the element of the $n$th refinement $\alpha_0^n$ 
  that contains $y_0$.  This is well-defined for all $n\ge1$ for 
   $\mu_{\bar Y}$-a.e.~$y_0\in\bar Y$.
   Moreover, $\pi(W^u_F(y))\cap\alpha_0^n(y_0)$ is a neighborhood of $y_0$ for all  $n\ge1$ for $\mu_Y$-a.e.~$y$.
   Since $\bar F$ is full-branch, $\bar F^n(\alpha_0^n(y_0))=\bar Y$
   for all $n$.

  By the Poincar\'e Recurrence Theorem, we may assume without
  loss that $y$ is recurrent: there
  exists $n_i\to\infty$ such that $F^{n_i}y\to y$.
  Therefore, for all
  large enough $i$ we have $\bar F^{n_i}(y_0)\in \pi(W^u_F(y))$ and
  hence the iterate of a connected piece of the unstable
  manifold of $y$ defined by
  \begin{align*}
    W_{n_i}=F^{n_i}\big( (\pi\mid W^u_F(y))^{-1}
    \alpha_0^{n_i}(y_0)\big)
  \end{align*}
  is a $u$-curve that crosses $Y$.
  The sequence $W_{n_i}$ has a convergent subsequence to $W$ by the
  Arzel\'a-Ascoli Theorem and by the recurrence
  assumption on $y$ we have $y\in W$.

  We claim that $W=W^u_{loc,f}(y)$, which completes the
  proof that $\mu_Y$-almost every point has an unstable
  manifold crossing $Y$. The
  last statement of the proposition is a simple restatement
  of this conclusion.

  To prove the claim, we consider $y'\in W$ and
  sequences $y_i,y'_i\in W_{n_i}$ such that
  $(y_i,y'_i)\to (y,y')$. Fix $l\ge1$ and choose
  $L\in\Z^+$ so that $\tau_{n_i}(y')>l$ for all
  $i\ge L$.
 By the definition of $W_{n_i}$ and
  since $W^u_{loc,f}(y')=W^u_F(y')$, we have uniform
  backwards contraction.  Thus
  \begin{align}\label{eq:cont-u-curves}
    |y_i-y'_i|
    &=
    |f^l(f^{\tau_{n_i}(y')-l}(z_i))-f^l(f^{\tau_{n_i}(y')-l}(y'))|
    \ge
    \frac{\lambda^{-l}}{c_0'}
    |f^{\tau_{n_i}(y')-l}(z_i)-f^{\tau_{n_i}(y')-l}(y')|
  \end{align}
  where $z_i\in W^u_{loc,f}(y')$ is such that
  $y_i=F^{n_i}(z_i)$. Hence $|f^{-l}(y_i)-f^{-l}(y'_i)|\le
  c_0'\lambda^{l} |y_i-y'_i|$. To obtain $c_0'$ we have used
  that all the iterates $W_n$ of $W^u_{loc,f}(y')$ are
  $u$-curves and so their length is comparable to the length
  of their projection $\pi(W_n)$ on $\bar X$; and then take
  advantage of the backward contraction property
  (\ref{eq:backcontraction}) associated to the partition
  $\alpha_0$ with the same contraction rate
  $\lambda$. Finally, since these constants are independent
  of $i$, letting $i\to\infty$ gives
  $|f^{-l}(y)-f^{-l}(y')|\le c_0'\lambda^{l} |y-y'|$ for each
  given fixed $l\ge1$.  This completes the proof of the
  claim and finishes the proof of the proposition.
\end{proof}

Using this geometric structure we can also prove the following:

\begin{proposition} \label{prop-localprod} The induced map
  $F:Y\to Y$ has a \emph{local product structure}: for any
  partition element $a\in\alpha$ there exists a measurable
  map $[\cdot,\cdot]:Y\times a\to a$ defined for all $y'\in
  a$ and $\mu_Y$ almost every $y\in Y$ such that 
 $$
 [y,y']\in W_F^u(y)\pitchfork W_F^s(y')
 $$
 consists of a unique point.  In addition, the map
 $[\cdot,\cdot]$ is constant along unstable manifolds in the
 first coordinate, and constant along stable manifolds in
 the second coordinate.  Furthermore, $[\cdot,\cdot]$ is
 $C^{1+\epsilon}$ in the second coordinate.
\end{proposition}

\begin{proof}
  From Proposition~\ref{pr:Wu_crosses} we have that for
  $\mu_Y$ almost every point $y$ the local unstable manifold
  $W^u_F(y)$ crosses $Y$.
  From the definition of geometric Lorenz attractor,
  $W^s_F(y')$ crosses $a$ transversely to
  $W^u_F(y)$, for every $y'\in Y$.
  Hence $[y,y']$ is well defined for $\mu_Y$-almost every
  $y\in Y$ and every $y'\in Y$.  Since $a$ is a union of local stable manifolds, it is immediate that if $y'\in a$ then $[y,y']\in a$.  

We note that if $[y,y']$
  is defined, then
  \begin{align*}
    w\in W^u_F(y)\mapsto [w,y']=[y,y']
    \qand
    w\in W^s_F(y')\mapsto [w,s]=[y,y']
  \end{align*}
  which shows that $[y,y']$ is constant along unstable
  manifolds on the first coordinate and stable manifolds on
  the second coordinate. In addition, the stable manifolds
  $W^s_F(y')$ depend continuously in the
  $C^{1+\epsilon}$ topology on the base point $y'$ (by the
  partial hyperbolicity of the attractor) and the unstable
  manifolds $W^u_F(y)$ depend measurably on $y$ (by
  nonuniform hyperbolicity). Hence $[\cdot,\cdot]$ is a
  measurable map and is $C^{1+\epsilon}$ along the second
  coordinate.
\end{proof}

\begin{proposition} \label{prop-periodic}
Suppose that $y,y'\in Y$ are such that $[y,y']$
 is well-defined.  Then, there is a sequence of periodic
 points $z_i\in Y$ for $F$ such that (i) $z_i\to y$, (ii)
 $[z_i,y']$ is well-defined for all $i$, and (iii)
 $[z_i,y']\to [y,y']$.
\end{proposition}

\begin{proof}
  We use Proposition~\ref{pr:Wu_crosses}: we can assume
  without loss that $y$ is
  recurrent. Let us fix a neighborhood $U$ of $y$ given by
  $U_1\times U_2$, where $U_1$ is an open subinterval of
  $[-1,1]\setminus\{0\}$ and $U_2$ is an open subset of
  $[-1,1]$. We fix a similar smaller neighborhood
  $V=V_1\times V_2$ such that closure of $V_j$ is contained
  in $U_j, j=1,2$. We can regard $V_1$ as a neighborhood
  of $\pi y$.

  In our setting, this ensures the existence of a sequence
  $n_i\to\infty$ such that $F^{n_i}y\to y$,
  $\pi\alpha^{n_i}(y)$ is a neighborhood of $\pi y$ and
  $\pi F^{n_i} (\alpha^{n_i}(y)) = \bar Y$; see the proof of
  Proposition~\ref{pr:Wu_crosses}.

  \begin{figure}[htpb]
    \centering
    \includegraphics[width=7.5cm]{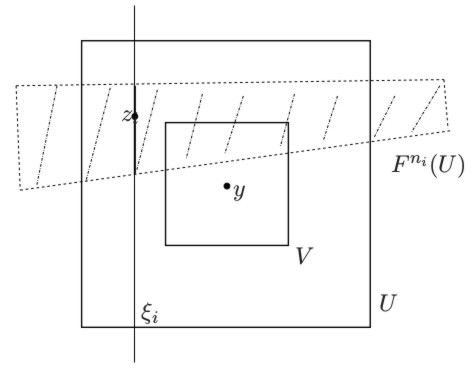}
    \caption{\label{fig:perdense} The density of periodic
      points for $F$.}
  \end{figure}

  Hence, there exists a stable leaf $\xi_i=\pi^{-1}(\bar
  x_i)\subset\alpha^{n_i}(y)$ for some $\bar x_i\in
  \pi\alpha^{n_i}(y)$ which is sent inside itself by
  $F^{n_i}$, by the uniform contraction of the stable leaves
  of $\cW_F^s$. Since we can assume without loss that
  $F^{n_i}y\in V$, then taking $n_i$ big enough, we claim
  that $F^{n_i}(\xi_i\cap U)\subset U$. Indeed, due to
  the domination assumption on $f$, the unstable manifold
  $W^u_F(F^{n_i}y)$ crosses $Y$ and its angle with
  respect to the horizontal direction is uniformly bounded
  from above, so
  $W^u_F(F^{n_i}y)\pitchfork\xi_i\subset U$, and
  the claim follows by uniform contraction of the stable
  leaves; see Figure~\ref{fig:perdense}.

  Thus we have a fixed point $z_i$ of $F^{n_i}$ in
  $\xi_i\cap U$. Since $U$ belongs to a fundamental
  system of neighborhoods of $y$, this proves item (i) in
  the statement of the Proposition.

  All periodic points of $f$ are hyperbolic of saddle type,
  thus $W^u_F(z_i)$ is well-defined and $\pi
  W^u_F(z_i)$ is a neighborhood of $\pi z_i=\bar x_i$.

  If $\pi W^u_F(z_i)\supset\pi\alpha^{n_i}(y)$, then
  clearly $ W^u_F(z_i)$ crosses $a$ and so $[z_i,y']$
  is well-defined. We claim that this is always the
  case. For otherwise, if $\pi
  W^u_F(z_i)\subset\alpha^{n_i}(y)$, then since $\bar
  F^{n_i}\mid\pi \alpha^{n_i}(y):\pi \alpha^{n_i}(y)\to
  \pi\alpha(y)$ is an expanding diffeomorphism and
  $W^u_F(z_i)$ is $F^{n_i}$-invariant, the length of
  $\bar F^{k n_i}(\pi W^u_F(z_i)), k\ge1$ grows while
  this set is contained in $\alpha^{n_i}(y)$. Thus $\pi
  W^u_F(z_i)$ covers $\pi\alpha^{n_i}(y)$, as claimed.

  Finally, for the continuity statement (iii), since
  $z_i\to y$ as $i\to\infty$ with $z_i$ periodic points of
  $F$, it is enough to show that $W^u_F(z_i)\cap a \to
  W^u_F(y)\cap a$ as smooth curves that cross
  $a$. Since each curve $W^u_F(z_i)\cap a$ is a
  $u$-curve, then there exists a accumulation point $\gamma$
  which is also a $u$-curve (by the Arzel\'a-Ascoli Theorem). We
  show that $\gamma$ contains $W^u_F(y)\cap a$.

  Indeed, $\gamma$ contains $y$.  By uniform backward
  contraction, if $z_i^0,\tilde z_i^0\in
  W^u_F(z_i)\cap a$ converge to $z,\tilde
  z\in\gamma\cap a$, then we can argue similarly to
  (\ref{eq:cont-u-curves}) since there are $z_i^k,\tilde
  z_i^k\in W^u_F(z_i)\cap a$ so that
  $z_i^0=F^k(z_i^k)$ and $\tilde z_i^0=F^k(\tilde z_i^k),
  k\ge1$. Hence, for a given fixed $k$ we get
  \begin{align*}
    |z_i^0-\tilde z_i^0|
    =
    |F^k(z_i^k)-F^k(\tilde z_i^k)|
    \ge
    \frac{\lambda^{-k}}{c_0'} |z_i^k -\tilde z_i^k|,
  \end{align*}
  and for limit points $z^k,\tilde z^k\in\gamma$ of
  $(z_i^k)_{i\ge1}$ and $(\tilde z_i^k)_{i\ge1}$ letting
  $i\to\infty$, we obtain (since $F$ is smooth in $a$)
  \begin{align*}
    |z-\tilde z|
    =
    |F^k(z_k)-F^k(\tilde z_k)|
    \ge
    \frac{\lambda^{-k}}{c_0'} |z_k-\tilde z_k|.
  \end{align*}
  Hence, because $k\ge1$ was arbitrary, we see that
  $\gamma\subset W^u_F(y)$, as needed. The proof is
  complete.
\end{proof}

\subsection{The induced roof function $R$}
\label{sec:induced-roof-functi}

We define the \emph{induced roof function} $R: Y\to\R^+$,
$R(y)=\sum_{\ell=0}^{\tau(y)-1}r(f^\ell y)$.  Since $r$, and
hence $R$, is constant along stable leaves, we also denote
by $R$ the quotient induced roof function $R:\bar Y\to\R^+$.

It follows in a completely analogous way to \cite[Section
4.2.2]{ArVar} that $R:\bar Y\to\R^+$ satisfies
\begin{align}\label{eq:bdd-norm-roof}
  \sup_{h\in\cH}\sup_{y\in\bar Y}|D(R\circ h)(y)|<\infty
\end{align}
where $\cH$ is the set of all inverse branches of $\bar
F:a_0\to\bar Y, a_0\in\alpha_0$. 

Indeed, let $h\in\cH$, $h:\bar Y\to a_0$ be an inverse
branch of $\bar F$ with inducing time $l=\tau(a_0)\ge1$ and
let us fix $y\in a_0$. Then 
\begin{align*}
  |D(R\circ h)(y)|
  &=
  |DR ( h(y) )|\cdot |Dh(y)|
  =
  \frac{|DR ( h(y) )|}{|D\bar F ( h(y) ) |}
  =
  \left|
    \sum_{i=0}^{l-1} \frac{(Dr\circ \bar f^i)\cdot D\bar f^i
    }{D\bar F}\circ h(y)
    \right|.
\end{align*}
Recall from the construction of the inducing partition using
hyperbolic times that conditions
\eqref{eq:backcontraction} and \eqref{eq:slow-recurrence} are satisfied
(here we require these conditions only with $n=1$).
By~\eqref{eq:backcontraction},
 \begin{align*}
 \left|\frac{D\bar f^i}{D\bar F}\right|\circ h(y) 
  \le c_0\lambda^{l-i},\quad i=0,\dots,l-1.
 \end{align*}
Moreover, by~\eqref{eq:slow-recurrence},
$|\bar f^i(h(y))|\ge \sqrt\lambda^{l-i}$ 
and so, by Lemma~\ref{le:prop2.6}(2), 
\begin{align*}
|(Dr\circ \bar f^i)\circ h(y)|\le c_1/\sqrt\lambda^{l-i}, \quad
i=0,\dots,l-1, 
\end{align*}
for some $c_1>0$ depending only on $f$. 
Altogether this implies that $ |D(R\circ h)(y)| \le
c_1\sum_{i=0}^\infty\lambda^{i/2}<\infty$ establishing~\eqref{eq:bdd-norm-roof}.

\begin{proposition} \label{prop-Rdecay}
There exists $c>0$ such that
$\mu_Y(R>t)=O(e^{-ct})$.
\end{proposition}

\begin{proof} By~\cite[Section 4.2.1]{ArVar}, ${\rm
    Leb}(R>t)=O(e^{-ct})$.  The result follows since
  $d\mu_Y/d{\rm Leb}$ is bounded, see for
  example~\cite[Proposition~4.5]{ArVar}.  (We caution that
  our $R:Y\to\R^+$ is denoted by $r:\Delta\to\R^+$
  in~\cite{ArVar}).
\end{proof}

 \begin{remark}   \label{rmk-Rdecay}
 Proposition~\ref{prop-Rdecay} is not necessary for the results in this paper but simplifies the exposition, see Remark~\ref{rmk-simplify}. 
\end{remark}

For $y,y'\in \bar Y$ define the \emph{separation time}
$s(y,y')$ to be the least integer $n\ge0$ such that $\bar
F^n(y),\bar F^n(y')$ are in distinct partition elements of
$\alpha_0$.  For any given $\theta\in(0,1)$ we define the
\emph{symbolic metric} $d_\theta(y,y')=\theta^{s(y,y')}$ on
$\bar Y$.  Let $|R|_\theta=\sup_{y\neq
  y'}|R(y)-R(y')|/d_\theta(y,y')$ denote the Lipschitz
constant of the quotient induced roof function $R:\bar
Y\to\R^+$ with respect to $d_\theta$.

  We write $r_k(y)$ for the sum $\sum_{i=0}^{k-1}r(f^i(y))$
  in what follows.

\begin{lemma}\label{le:Lip-symbolic}
  There exists $C>0$ such that for all $y,y'\in \bar Y$ with
  $s(y,y')\ge1$ and $0\le k\le \tau(y)=\tau(y')$ we have
  $ |r_k(y)-r_k(y')|\le C|\bar F(y)-\bar
  F(y')|^\epsilon.$ As a consequence, there exists
  $\theta\in(0,1)$ such that $|R|_\theta<\infty$ and also
  $|\bar F(y)-\bar F(y')|\le C d_\theta(y,y')$.
\end{lemma}

\begin{proof}
  Let us consider $y,y'\in \bar Y$ such that
  $s(y,y')=n\ge1$. Then $y'\in\alpha_0^n(y)$ and so
  $\tau(\bar F^i(y))=\tau(\bar F^i(y')), i=0,\dots,
  n-1$. Thus, from the choice of the cross-section, ensuring
  that $r$ is constant on stable leafs, together with
  Lemma~\ref{le:prop2.6}(2), we can write 
  \begin{align*}
    R(y)-R(y')
    &=
    \sum_{\ell=0}^{\tau(y)-1}[r(\bar f^\ell(y))-r(\bar f^\ell(y'))]
    \\
    &=
    \sum_{\ell=0}^{\tau(y)-1}
    \big[
    -\lambda_u^{-1}(\log|\bar f^\ell(y)|-\log|\bar f^\ell(y')|)
    +
    h(\bar f^\ell(y))-h(\bar f^\ell(y'))
    \big].
  \end{align*}
Combining (\ref{eq:backcontraction}) together
with~(\ref{eq:slow-recurrence}) we obtain
\begin{align*}
  |R(y)-R(y')|
  &\le
  C \sum_{\ell=0}^{\tau(y)-1}
  \left[
    \frac{|\bar f^\ell(y)-\bar f^\ell(y')|}{\max\{|\bar f^\ell(y)|,|\bar f^\ell(y')|\}}
      + \|h\|_\epsilon|\bar f^\ell(y)-\bar f^\ell(y')|^\epsilon
  \right]
  \\
  &\le
  C \sum_{\ell=0}^{\tau(y)-1}
  \left[
    c_0\frac{\lambda^{\tau(y)-\ell}}{\sqrt{\lambda}^{\tau(y)-\ell}}
    |\bar F(y)-\bar F(y')|
    + \|h\|_\epsilon \lambda^{\epsilon(\tau(y)-\ell)}
  |\bar F(y)-\bar F(y')|^\epsilon
    \right]
  \\
  &\le
  \kappa\lambda^{\epsilon/2}\cdot
  \frac{1-\lambda^{\frac{\epsilon}2\cdot\tau(y)}}{1-\lambda^{\epsilon/2}} 
  |\bar F(y)-\bar F(y')|^\epsilon,
\end{align*}
for some constant $\kappa>0$ depending on the flow
only. This implies in particular the first statement of the
lemma.

Finally, because $\tau$ is at least $1$, we also have
\begin{align*}
  |\bar F(y)-\bar F(y')|^\epsilon
  \le
  c_0\lambda^{\epsilon \tau_{n-1}(\bar F(y))}
  |\bar F^n(y)-\bar F^n(y')|^\epsilon
  \le
  c_0\lambda^{\epsilon (n-1)}|\bar F^n(y)-\bar F^n(y')|^\epsilon
\end{align*}
which, combined with the previous inequality, gives
another constant $K>0$ depending only on the flow satisfying
\begin{align*}
  |R(y)-R(y')|
  &\le
  K\lambda^{n\epsilon/2}|\bar F^n(y)-\bar F^n(y')|^\epsilon.
\end{align*}
We can find $\lambda^{\epsilon/2}<\theta<1$ and a
constant $C_0>0$ so that $K\lambda^{n\epsilon/2} \le
C_0\theta^{n} = C_0 d_\theta(y,y')$ for all $n\ge1$
and, since $|\bar F^n(y)-\bar F^n(y')|^\epsilon$ is
bounded above, the proof is complete.
\end{proof}

\subsection{Expansion for the flow}
\label{sec:expansion-flow}

We are now ready to prove a useful consequence of
backward contraction for the quotient map and expansion
of the flow in the linearizable region. We keep the
choice of $\theta$ from Lemma~\ref{le:Lip-symbolic}.
Also, we define $d_\theta(y,y')$ for points $y,y'\in Y$ by setting
$d_\theta(y,y')=d_\theta(\pi y,\pi y')$.

\begin{lemma}
  \label{le:expansion-flow}
  There exist constants $C,\kappa>0$ such that for all $y,y'\in Y$
  satisfying $d_\theta(y,y')<\kappa$ and
  all $u\in(0,\min\{R(y),R(y')\})$, we have
$|Z_u(y)-Z_u(y')|\le C d_\theta(y,y').$
\end{lemma}

\begin{proof}
Taking $\kappa<1$, we have
  $\tau(y)=\tau(y')$.  There exist
  $k,k'\in\{1,\dots,\tau(y)\}$ such that
  \begin{align*}
    u\in[r_{k-1}(y),r_k(y)] \quad\text{and}\quad
    u\in[r_{k'-1}(y'),r_{k'}(y')].
  \end{align*}
  From Lemma~\ref{le:Lip-symbolic} we know that
  $|r_k(y)-r_k(y')|\le C d_\theta(y,y')$. Hence 
  for $d_\theta(y,y')$ small
  enough, $|r_k(y)-r_k(y')|<\inf r$ for all $0\le
  k\le\tau(y)$.  It follows that $|k-k'|\le 1$, and we
  may suppose that $k\ge k'$. Hence, there are two cases to
  consider.
  \begin{figure}[htpb]
    \centering
    \includegraphics[width=8.5cm]{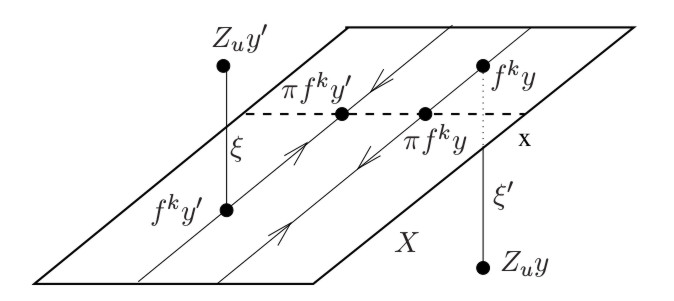}
    \caption{\label{fig:case_kprime} Estimating the
      distance between $Z_uy$ and $Z_uy'$.}
  \end{figure}
  \paragraph{\bf The case $k=k'+1$:} We have $r_{k'}(y)\le u\le r_{k'}(y')$
    and so the orbit of $y$ has already had $k'$ returns to
    the cross-section $X$, while the orbit of $y'$ has only
    returned $k'-1$ times. We estimate the distance
    between the points with the distance between the $k'th$
    returns of both orbits to $X$ as follows.  Setting
    $\xi=u-r_{k'}(y)$ and $\xi'=r_{k'}(y')-u$, we get
    \begin{align*}
      |Z_u y - Z_u y'|
      &\le
      |\xi+\xi'|+|f^ky-f^ky'|
      \\
      &\le
      |r_{k'}(y)-r_{k'}(y')|+|f^ky-\bar f^k\pi y|
      +|\bar f^k\pi y - \bar f^k\pi y'|
      +|\bar f^k\pi y' - f^ky'|
      \\
      &\le
      C d_\theta(y,y')
      + c\lambda^k(|y- \pi y|+|y'- \pi y'|)
      +c_0\lambda^{\tau(y)-k}|\bar Fy-\bar Fy'|
    \end{align*}
    where we have used the uniform contraction of the stable
    foliation of the attractor,
together with Lemma~\ref{le:Lip-symbolic}
    and the uniform backward contraction of iterates of
    $\bar f$; see Figure~\ref{fig:case_kprime}. Again from
    Lemma~\ref{le:Lip-symbolic} and the choice
    $\lambda^{\epsilon/2}<\theta<1$ we obtain
    \begin{align*}
      |Z_u y - Z_u y'| 
      \le 
      (C+2\ell +c_0 C) d_\theta(y,y')
    \end{align*}
    where $\ell$ is the length of the largest stable leaf in
    the cross-section $X$.
  \paragraph{\bf The case $k=k'$:} Now both points are past their
    $(k-1)$'th return and we again estimate the distance
    comparing with the distance of their $(k-1)$'th returns.
    Setting $\xi=u-r_{k-1}(y)$ and
    $\xi'=u-r_{k-1}(y')$ and assuming without loss that
    $\xi'\ge\xi$ we get
    \begin{align*}
      |Z_u y - Z_u y'|
      &\le
      |\xi-\xi'|+|Z_\xi(f^{k-1}y)-Z_\xi(f^{k-1}y')|
    \end{align*}
    and $|\xi-\xi'|\le Cd_\delta(y,y')$ as before, while
\begin{align*}
  |Z_\xi(f^{k-1}y)-Z_\xi(f^{k-1}y')|
  &\le
  |Z_\xi(f^{k-1}y)-Z_\xi(\pi f^{k-1}y)|
  \\
  &\quad+|Z_\xi(\pi f^{k-1}y)-Z_\xi(\pi f^{k-1}y')|
  +  |Z_\xi(\pi f^{k-1}y)-Z_\xi( f^{k-1}y')|.
\end{align*}
The uniform contraction along stable leaves and the relation
$\pi f = \bar f\pi$ allows us to write
\begin{align*}
  |Z_\xi(f^{k-1}y)-Z_\xi(f^{k-1}y')|
  &\le
  c\lambda^\xi\cdot c\lambda^{k-1} \cdot 2\ell
  + |Z_\xi(\pi f^{k-1}y)-Z_\xi(\pi  f^{k-1}y')|
\end{align*}
and since $\lambda^{k-1}\le d_\theta(y,y')$, we are left
to prove that
\begin{align}\label{eq:expansionflow}
   |Z_\xi(\pi f^{k-1}y)-Z_\xi(\pi  f^{k-1}y')|\le Cd_\theta(y,y').
\end{align}
For this we use the construction of the geometric
Lorenz attractor with the linearizable Lorenz-like
singularity to explicitly calculate trajectories. In
this way, we easily see that the distance between the
pair of stable leaves $\zeta=\pi^{-1}(\pi f^{k-1}y)$
and $\zeta'=\pi^{-1}(\pi f^{k-1}y')$ (on $X$) is
expanded by $e^{\lambda_1 t}$, that is,
\begin{align*}
  d(Z_t\zeta,Z_t\zeta')\ge e^{\lambda_1 t} d(\zeta,\zeta'),
\end{align*}
as long as $Z_s(\pi f^{k-1}y)$ and $Z_s(\pi f^{k-1}y')$
remain in the linearizable region around the
singularity, for $0\le s\le t$.

The flight time from $X$ to $X'$ (the boundary of the
linearizable region) is given by $-\log|\pi f^{k-1}y|$
and $-\log|\pi f^{k-1}y'|$, and their difference is
uniformly bounded since, by the backward contraction
(\ref{eq:backcontraction}) and
slow recurrence (\ref{eq:slow-recurrence}) properties,
\begin{align*}
 \left |\log\frac{|\pi f^{k-1}y|}{|\pi f^{k-1}y'|}\right|
 &\le
 \frac{|\bar f^{k-1}(\pi y)-\bar f^{k-1}(\pi y')|}
 {\max\{|\bar f^{k-1}(\pi y)|,|\bar f^{k-1}(\pi y')|\}}
 \le
 c_0\frac{\lambda^{\tau(y)-k+1}}{\sqrt{\lambda}^{\tau(y)-k+1}}
 =c_0\lambda^{(\tau(y)-k+1)/2}.
\end{align*}
Hence, we have expansion of the distance between
$\zeta,\zeta'$ in the linear region, and the flow from
$X'$ back to $X$ is performed in a uniformly bounded
time for all points of the attractor. Thus, this last
non-linear action of the flow distorts the distance by
at most some constant factor (a bound on the norm of
the derivative of the flow on a bounded interval of
time). Therefore, we have shown that the left hand side of
(\ref{eq:expansionflow}) is bounded by a constant
factor of $|\bar f^ky-\bar f^ky'|$. This last difference
is bounded by $c_0\lambda^{\tau(y)-k}|\bar Fy-\bar
Fy'|$ which is bounded by the expression on the right
hand side of (\ref{eq:expansionflow}).
The proof is complete.
\end{proof}

\subsection{Suspension flow}
\label{sec-suspension}

In Subsection~\ref{sec:lorenz-attractors-as}, we saw that the
geometric Lorenz flow can be modelled as a suspension flow 
$S_t:X^r\to X^r$ where $X$ is the Poincar\'e section and $r$ is the first hit time.

Shrinking the cross-section to $Y$ and using the induced roof function
$R$ (which need not be the first hit time), we have the alternative model
of the geometric Lorenz flow as the suspension flow
$S_t:Y^R\to Y^R$ 
over a uniformly hyperbolic map $F:Y\to Y$ with 
integrable but unbounded return time function $R:Y\to\R^+$.
Again the suspension flow is given by $S_t(y,u)=(y,u+t)$ computed
modulo identifications, and  the probability measure
$\mu=\mu_Y\times{\rm Lebesgue}/\int R\,d\mu_Y$ is $S_t$
invariant.

Similarly, we define the
quotient suspension semiflow $\bar S_t:\bar Y^R\to\bar Y^R$ with
invariant probability measure
$\bar\mu=
\mu_{\bar Y}\times{\rm Lebesgue}/\int R\,d\mu_{\bar Y}$.

The next result enables us to pass from the ambient manifold $\R^3$ to $Y^R$ by means of the projection $p:Y^R\to \R^3$, $p(y,u)=Z_uy$.

\begin{proposition} \label{prop-p}
Let $(y,u),\,(y',u')\in Y^R$.   Then 
\[
|p(y,u)-p(y',u')|\le C\{d_\theta(y,y')+|u-u'|\}.
\]
\end{proposition}

\begin{proof}
Without loss, we can suppose that $u\le u'$.
By the mean value theorem, there is a $u''$ between $u$ and $u'$ such that
\begin{align*}
|p(y',u)-p(y',u')|=|Z_uy'-Z_{u'}y'|\le 
|\partial_t Z_t(y')|_{t=u''}||u-u'|=|G(Z_{u''}y')||u-u'|,
\end{align*}
where $G$ is the underlying vector field.  Since $G$ is continuous and we are restricting to $y$ lying in the compact attractor $\Lambda$, we obtain that there is a constant $C>0$ such that
$|p(y',u)-p(y',u')|\le  C|u-u'|$.

Also by Lemma~\ref{le:expansion-flow},
$|p(y,u)-p(y',u)|  = |Z_u(y)-Z_u(y')|\le Cd_\theta(y,y')$.

The result follows by the triangle inequality.
\end{proof}


%

\section{Temporal distortion function}
\label{sec:tempor-distort-funct}

In this section, we introduce the temporal distortion
function and prove a result about the dimension of its
range.

For all
$y,z\in Y$ belonging to the same unstable manifold for
$F:Y\to Y$, we define
\begin{align*}
  D_0(y,z)=\sum_{j=1}^\infty [r(f^{-j}y)-r(f^{-j}z)].
\end{align*}

We remark that each term in the sum makes sense since
$f$ is invertible on the attractor. Moreover we note that
property (C) ensures that the roof function can be seen as a
$C^{1+\epsilon}$ function on $\bar X$ with a logarithmic
singularity at $0$. We now prove that $D_0$ is
well-defined.

\begin{lemma}\label{lem-D0}  
  The function $D_0$ is measurable and $D_0(y,z)$ is finite
  for $\mu_Y$-almost every $y$ and every $z\in
  W_F^u(y)$.

  Moreover, $D_0$ is continuous in the following sense.
  Suppose that $D_0(y,z)$ is well-defined and
  $\epsilon>0$ is given. Then, there exists $\delta>0$
  such that $|D_0(y',z')-D_0(y,z)|<\epsilon$ for all
  pairs $(y',z')$ satisfying (i) $D_0(y',z')$ is
  well-defined, (ii) $|y'-y|<\delta$, $|z'-z|<\delta$,
  and (iii) $d_\theta(y',y)<\delta$,
  $d_\theta(z',z)<\delta$.
\end{lemma}

\begin{proof}  
Although
  the iterates $f^{-j}y$, $f^{-j}z$ are close by backward contraction, the values $r(f^{-j}x)$ and $r(f^{-j}y)$ need not be close. 
Hence, we consider the induced map
  $F:Y\to Y$ and the induced roof function $R:Y\to\R^+$
  given by $R(y)=\sum_{\ell=0}^{\tau(y)-1}r(f^\ell
  y)$.  Note however that $F$ is not invertible (unlike
  $f$) so some care is needed in the following
  argument.

Write $y_0=y$, $z_0=z$.
By ergodicity of $\mu_Y$ under $F$,  we may suppose that there exist $y_i\in Y$, and $a_i\in\alpha$, $i\ge1$, such that $y_i\in a_i$ and $Fy_i=y_{i-1}$
for all $i\ge1$.
Since $\bar F$ is full branch, $F(a_1\cap W^u_F(y_1))$ covers $Y$
and in particular covers $W^u_F(y_0)$.
Hence there exists $z_1\in a_1\cap W^u_F(y_0)$ such that $Fz_1=z_0$.
  Inductively, we obtain $z_i\in a_i\cap W^u_F(y_i)$, $i\ge1$, such that
  $Fz_i=z_{i-1}$.  

  By construction, $y_i=f^{-\tau(y_i)}y_{i-1}$.
  Inductively,
  $y_i=f^{-(\tau(y_1)+\dots+\tau(y_i))}y$.  Hence
\begin{align*}
  R(y_i)=\sum_{\ell=0}^{\tau(y_i)-1}r(f^\ell
  f^{-(\tau(y_1)+\dots+\tau(y_i))}y)
  =\sum_{\ell=\tau(y_1)+\dots+\tau(y_{i-1})+1}^{\tau(y_1)+\dots+\tau(y_i)}r(f^{-\ell}
  y).
\end{align*}
Formally summing up the contributions from $y_i$ and similarly $z_i$, we obtain the equivalent definition
  $D_0(y,z)= \sum_{i=1}^\infty [R(y_i)-R(z_i)]$.
  Moreover, since $R$ is constant along stable leaves,
  writing $\bar y_i=\pi y_i,\bar z_i=\pi z_i$, 
  \begin{align} \label{eq-altD0}
    D_0(y,z)=
    \sum_{i=1}^\infty [R(\bar y_i)-R(\bar z_i)].
  \end{align}

  To justify these formal manipulations, it suffices to
  prove that the series in~\eqref{eq-altD0} converges.
  In the process, we verify the first statement of the
  lemma.  Recall from Lemma~\ref{le:Lip-symbolic} that
  we can choose $\theta\in(0,1)$ so that $R$ is
  $d_\theta$-Lipschitz with Lipschitz constant
  $|R|_\theta$.  We have $s( y_i,z_i)=i+s(y,z)$ for
  $i\ge0$ and so
\begin{align*}
  |D_0(y,z)|
  & \le
  \sum_{i=1}^\infty |R(\bar y_i)-R(\bar z_i)|
  \le
  \sum_{i=1}^\infty |R|_\theta d_\theta(y_i,z_i)
  \\ 
  & = |R|_\theta\sum_{i=1}^\infty\theta^i d_\theta(
  y_0,z_0)
  = |R|_\theta \theta(1-\theta)^{-1} d_\theta(y,z)
  <\infty
\end{align*}
as required.

It remains to prove the last statement of the lemma.
Let $N\ge1$.  By the above argument,
\[
D_0(y,z)-D_0(y',z')=A(y,z)-A(y',z')+B(y,y')-B(z,z'),
\]
where
\[
A(y,z)=\sum_{i=N}^\infty|R(\bar y_i)-R(\bar z_i)|, \quad
B(y,y')=\sum_{i=0}^{N-1}|R(\bar y_i)-R(\bar y'_i)|.
\]
Moreover, $A(y,z),\,A(y',z')\le C \theta^N$.

For $(y',z')$ sufficiently close to $(y,z)$, the sequences of
partition elements $a_i$ containing $y_i,z_i$
and $a_i'$ containing $y'_i,z'_i$ coincide for $i=1,\dots,N$.
Hence $s(y_i,y_i')=i+s(y,y')$ and
$s(z_i,z_i')=i+s(z,z')$ for $i=1,\dots,N$ and so
\[
B(y,y')\le Cd_\theta(y,y'), \quad
B(z,z')\le Cd_\theta(z,z').
\]

Given $\epsilon>0$, we choose $N$ so that $C\theta^N < \epsilon/4$.
Then we choose $(y',z')$ so close to $(y,z)$ that 
(i) $a_i'=a_i$ for $i=1,\dots,N$,
(ii) $Cd_\theta(y,y')<\epsilon/4$, and
(iii) $Cd_\theta(z,z')<\epsilon/4$.
Then $|D_0(y,z)-D_0(y',z')|<\epsilon$ as required.
\end{proof}

\subsection{A double inducing scheme} \label{sec-double}
As in~\cite{LMP05}, the second iterate $y_1=f^2(0+)$ plays an important role in
establishing mixing properties.   With that in mind, we consider two
inducing schemes $F_i=f^{\tau_i}:Y_i\to Y_i$, $i=1,2$, whose quotients $\bar F_i:\bar Y_i\to\bar Y_i$ are $C^{1+\epsilon}$ piecewise expanding maps with full branches.  By the l.e.o.\ condition, we can choose $\bar Y_0$ and $\bar Y_1$ to be disjoint open intervals containing $0$ and $y_1$ respectively.
Setting $Y=Y_0\cup Y_1$, we obtain a combined (nonergodic) inducing scheme
$F=f^\tau:Y\to Y$ where $F|_{Y_i}=F_i$, $\tau|_{Y_i}=\tau_i$.  
The partition $\alpha$ on $Y$ is the union of the partitions on $Y_0$ and $Y_1$.

By Proposition~\ref{pr:Wu_crosses}, for almost every $y\in Y_0$ there is a local unstable manifold $W_F(y)$ that covers $Y_0$, and then by
Proposition~\ref{prop-localprod} the product $[y,y']$ is defined 
for every $y'\in Y_0$.   The same statement holds with $Y_0$ changed to $Y_1$.
In particular, Lemma~\ref{lem-D0} goes through for this inducing scheme by considering points in $Y_0$ and $Y_1$ separately.  (By convention, $D_0(y,z)$ is never defined for $y,z$ lying in distinct connected components of $Y$.)

Throughout most of the remainder of this section, until Subsection~\ref{sec-Hdim}, we work with this inducing scheme. 

\subsection{Young tower from the inducing scheme}
\label{sec:young-tower-from}

To the inducing scheme $F=f^\tau:Y\to Y$ constructed in Subsection~\ref{sec-double}, we associate the
Young tower~\cite{Yo98} $\hat f:\Delta\to\Delta$ where
$\Delta=\{(y,\ell):y\in Y,\,\ell=0,1,\dots,\tau(y)-1\}$
and $\hat f(y,\ell)=\begin{cases} (y,\ell+1), & \ell\le \tau(y)-2 \\
  (Fy,0), & \ell=\tau(y)-1\end{cases}$.  Note that $F=\hat
f^\tau$ is a first return map to $Y$.  The projection
$\pi:\Delta\to Y$, $\pi(y,\ell)=f^\ell y$, defines a
semiconjugacy between $\hat f:\Delta\to\Delta$ and $f:X\to
X$.  Let $\hat r=r\circ\pi:\Delta\to\R$ denote the lifted
roof function.  The partition $\alpha$ of $Y$ extends to a
partition
$\hat\alpha=\{a\times\ell:a\in\alpha,\,0\le\ell<\tau|_a\}$
of $\Delta$.
Let $\Delta_\ell=\{(y,\ell):y\in Y,\,0\le\ell<\tau(y)\}$
denote the $\ell$'th level of the tower.
We write $\Delta_\ell=\Delta_{\ell,0}\cup\Delta_{\ell,1}$ where
$\Delta_{\ell,i}=\{p=(y,\ell)\in\Delta_\ell:y\in Y_i\}$.

Fix $\ell\ge0$, $i\in\{0,1\}$.
For $p=(y,\ell)\in\Delta_{\ell,i}$ we define the stable and unstable manifolds
$W^s(p)=W^s_F(y)\times\ell$,
$W^u(p)=W^u_F(y)\times\ell\in\Delta_{\ell,i}$.
For $p=(y,\ell)$, $q=(z,\ell)\in\Delta_{\ell,i}$ we define $[p,q]=([y,z],\ell)
\in\Delta_{\ell,i}$.
Again if $q\in \hat a$ for some $\hat a\in\hat\alpha$, then $[p,q]\in \hat a$.

We say that $p,q\in\Delta$ lie in the same unstable manifold if
$p=(y,\ell)$, $q=(z,\ell)$ lie in $\Delta_{\ell,i}$ for some $\ell,i$, and
$y,z$ lie in the same unstable manifold.
In that case we define
\begin{align*}
  D_0(p,q)=\sum_{j=1}^\infty [\hat r(\hat f^{-j}p)-\hat r(\hat f^{-j}q)].
\end{align*}
Note that 
\begin{align} \label{eq-D0}
D_0(p,q)=D_0(y,z)+\sum_{j=0}^{\ell-1}[r(f^jy)-r(f^jz)],
\end{align}
so the considerations in Lemma~\ref{lem-D0} for $D_0$ restricted to
points in $Y$ apply also to $D_0$ on $\Delta$.

For $\hat a,\,\hat a'\in\hat\alpha$ with
$\hat a,\hat a'\subset\Delta_{\ell,i}$ for some $\ell,i$,
we define the {\em temporal distortion
  function} $D:\hat a\times \hat a'\to \R$ by setting
\begin{align*}
D(p,q)
& =
\sum_{j=-\infty}^\infty[ \hat r(\hat f^jp)-\hat r(\hat f^j[p,q])-\hat r(\hat f^j[q,p])+\hat r(\hat f^jq)],
\end{align*}
for $p\in \hat a$, $q\in \hat a'$.
We note that
\begin{align*}
D(p,q) 
& = \sum_{j=-\infty}^{-1}
\big[\hat r(\hat f^jp)-\hat r(\hat f^j[p,q])-\hat r(\hat f^j[q,p])+\hat r(\hat f^jq)\big] 
  = D_0(p,[p,q])+ D_0(q,[q,p]),
\end{align*}
where the first equality follows since $r$ is constant on
stable manifolds and the second is by definition of
$D_0$; see Figure~\ref{fig:def_D}.
Hence, $D$ is almost everywhere well defined by
Proposition~\ref{prop-localprod},
Lemma~\ref{lem-D0} and~\eqref{eq-D0}. 

  \begin{figure}[htpb]
    \centering
    \includegraphics[width=8cm]{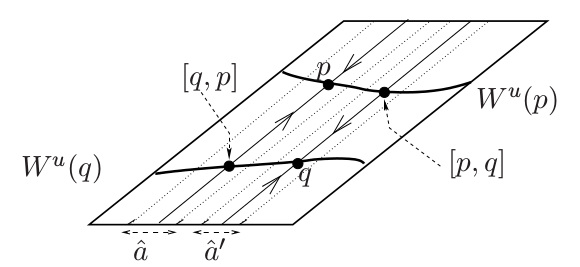}
    \caption{\label{fig:def_D} The definition of the
      temporal distortion function.}
  \end{figure}

\subsection{Integrability and locally constant roof functions}
This section follows
closely~\cite[Appendix]{dolgopyat98}. Our purpose is to show
that the temporal distortion function is non-zero for
geometric Lorenz attractors.

\begin{proposition} \label{prop-zero} 
Let $\hat a,\,\hat a'\in\hat\alpha$ with
$\hat a,\hat a'\subset\Delta_{\ell,i}$ for some $\ell,i$.
Suppose that  $D|_{\hat a\times \hat a'}\equiv0$. Then
  for all $p\in \hat a$, $q\in \hat a'$, the function $D_0(p,[p,q])$ is
  constant along the stable manifolds of
  $p$ and $q$.
\end{proposition}

\begin{proof} By Proposition~\ref{prop-localprod}, $w\mapsto
  [p,w]=[p,q]$ is constant along stable manifolds of $q$.
  Hence, $w\mapsto D_0(p,[p,w])$ is constant along the stable
  manifold of $q$.  Similarly, $w\mapsto D_0(q,[q,w])$ is
  constant along the stable manifold of $p$.  But $D|_{\hat a\times \hat a'}\equiv0$ implies
  that these two expressions are equal up to a minus sign
  and the result follows.
\end{proof}

For each $\hat a\in\hat\alpha$ with $\hat a\subset\Delta_{\ell,i}$, we associate a point $q_{\hat a}\in\Delta_{\ell,i}$.
Then $[p,q_{\hat a(p)}]$ is defined for almost every $p\in\Delta$
(here, $\hat a(p)$ is the partition element in $\hat\alpha$ containing $p$).
Define $\chi,u:\Delta\to\R$ by setting
\begin{align} \label{eq-chi}
\chi(p) & =D_0(p,[p,q_{\hat a(p)}])=\sum_{j=1}^\infty
\{\hat r(\hat f^{-j}p)-\hat r(\hat f^{-j}[p,q_{\hat a(p)}])\}, \\
\label{eq-u}
u(p) & = 
\sum_{j=1}^\infty
\{\hat r(\hat f^{-j}[\hat fp,q_{\hat a(Fp)}])
-\hat r(\hat f^{-j}[p,q_{\hat a(p)}])\}.
\end{align}
It follows from the definitions that
$\hat r=\chi\circ \hat f-\chi+u$ on $\Delta$.

\begin{proposition} \label{prop-u} If $D\equiv0$, then 
$\chi:\Delta\to\R$ is continuous (indeed $C^1$) on $Y\cong Y\times 0$ and
$u$ is constant on partition elements of $\Delta$.
\end{proposition}

\begin{proof}
By definition, $u$ is constant along local unstable manifolds $W^u(p)\cap a$
for all $p\in \hat a$, $\hat a\in\hat\alpha$.
But if $D\equiv0$, then by Proposition~\ref{prop-zero}
we have that $\chi$ is constant along stable manifolds.
Hence the same holds for $\chi\circ \hat f$.  But $\hat r$ is already
constant along stable manifolds, so we deduce that
$u=\hat r-\chi\circ \hat f+\chi$ is constant along stable manifolds.

We have shown that $u$ is constant along stable and
unstable manifolds and hence is constant on partition elements.

On $Y$, we obtain
$R=\sum_{\ell=0}^{\tau-1}\hat r\circ \hat f^\ell=\chi\circ F-\chi+\tilde u$,
where $\tilde u=\sum_{\ell=0}^{\tau-1}u\circ \hat f^\ell$ is
constant on partition elements.
Since $R$, $\chi$ and $u$ are constant along stable manifolds and hence are well-defined on $\bar Y$ we have that
$R=\chi\circ\bar F-\chi+\tilde u$ on $\bar Y$.

Restricting to $\bar Y_i$ for $i=0,1$, and recalling~\eqref{eq:bdd-norm-roof},
we note that
the map $\bar F_i:\bar Y_i\to\bar Y_i$ satisfies all the requirements
  of \cite[Proposition 7.4]{AvGoYoc} allowing us to conclude 
that $\chi_i|_{\bar Y_i}$ has a $C^1$ version.
Hence $\chi|_{\bar Y}$ has a $C^1$ version.
\end{proof}

However, this property contradicts the structure of
geometric Lorenz attractors, as follows.

\begin{theorem}   \label{thm-D}
For any geometric Lorenz flow, the temporal distortion
function $D$ is not identically zero.
\end{theorem}

\begin{proof}
We adapt the strategy in~\cite{LMP05}.  Let $y_1=f^2(0+)$.  Recall that
$0$ and $y_1$ lie in the interior of $\bar Y_0$ and $\bar Y_1$ respectively.
Choose disjoint subsets $U_0,U_1\subset Y$ that are the closure of unions of partition elements such that 
$W^s(y_1)\subset \Int U_1$ while $U_0$ contains a rectangle of the 
form $[0,\delta]\times [-1,1]$.
Shrinking $U_0$ if necessary, we can ensure that $\tau|U_0>2$ and that
$f^2U_0\subset U_1$.
Fix $z_*\in U_0$ and note that $f^2z_*\in U_1$.

Let $y\in U_0$. We claim that $f^2[y,z_*]=[f^2y,f^2z_*]$. 
It is clear that $f^2(W^s(z^*))\subset W^s(f^2z^*)$, so
$f^2[y,z^*]\in W^s(f^2z^*)$. We need to show that
$f^2[y,z^*]\in W^u_F(f^2y)$.

Now $W_F^u(y)\subset W_{loc,f}^u(y)$ and so
$f^2(W_F^u(y))\subset f^2( W_{loc,f}^u(y))$ and also
$W^u_F(f^2y)\subset W^u_{loc,f}(f^2y)\subset
f^2(W^u_{loc,f}(y))$. Moreover, both $f^2(W_F^u(y))$ and
$W^u_F(f^2y)$ cross $f^2U_0$, since $W^u_F(y)$ crosses
$Y_0\supset U_0$ and $W^u_F(f^2y)$ crosses $Y_1\supset f^2
U_0$. Therefore we conclude that $f^2[y,z^*]\in
f^2(W_F^u(y)\cap U_0)=W^u_F(f^2y)\cap f^2U_0$.
This proves the claim.

Define $\chi$ and $u$ as in~\eqref{eq-chi} and~\eqref{eq-u} stipulating
$q_{(a,2)}=(z_*,2)$  and $q_{(f^2a,0)}=(f^2z_*,0)$ for $a\in U_0$.

Suppose for contradiction that $D\equiv0$.
By Proposition~\ref{prop-u}, for $y\in Y$ we have
\begin{align} \label{eq-r2}
r(y)+r(fy)
=\hat r(y,0)+\hat r(y,1)=\chi(y,2)-\chi(y,0)+\tilde u(y,0),
\end{align}
where $\chi$ is continuous on $Y\cong Y\times 0$ and $\tilde u(y,0)=u(y,0)+u(y,1)$ is constant
on partition elements.

We claim that $\chi$ is continuous on $U_0\times 2$ and
that $\lim_{y\to0^+}\chi(y,2)=\chi(y_1)$.
It then follows from~\eqref{eq-r2} that $\tilde\ell$ is constant on $U_0\times 0$
and moreover that all terms in~\eqref{eq-r2} converge as $y\to0^+$ with the
exception of $r(y)$ which diverges to $+\infty$.   This is the desired contradiction.

It remains to verify the claim.
For $y\in a$, $a\subset U_0$ we compute that
\begin{align*}
[(y,2),q_{\hat a(y,2)}]=[(y,2),(z_*,2)]=([y,z_*],2)
\end{align*}
and
\begin{align*}
[(f^2y,0),q_{\hat a(f^2y,0)}]=[(f^2y,0),(f^2z_*,0)]
=([f^2y,f^2z_*],0)
\end{align*}
so that
\[
\pi[(y,2),q_{\hat a(y,2)}]=f^2[y,z_*]=[f^2y,f^2z_*]=\pi
[(f^2y,0),q_{\hat a(f^2y,0)}].
\]
Hence
\begin{align*}
\chi(y,2) & 
= \sum_{j=1}^\infty \bigl\{\hat r(\hat f^{-j}(y,2))-
\hat r(\hat f^{-j}[(y,2),q_{\hat a(y,2)}])\bigr\}
\\ & = \sum_{j=1}^\infty \bigl\{r(f^{-j}\pi(y,2))-
r(f^{-j}\pi[(y,2),q_{\hat a(y,2)}])\bigr\}
\\ & = \sum_{j=1}^\infty \bigl\{r(f^{-j}f^2y)-
r(f^{-j}\pi [(f^2y,0),q_{\hat a(f^2y,0)}]\bigr\}
\\ & = \sum_{j=1}^\infty \bigl\{\hat r(\hat f^{-j}(f^2y,0))-
\hat r(\hat f^{-j}[(f^2y,0),q_{\hat a(f^2y,0)}]\bigr\}
=\chi(f^2y,0).
\end{align*}
The claim follows from continuity of $\chi$ on $Y$.
\end{proof}


\subsection{Smoothness of the temporal distortion function}
\label{sec:dimens-range-tempor}

\begin{proposition} \label{prop-nonzero}
There exists $\hat a,\hat a'\in\hat \alpha$ with $\hat a,\hat a'\subset\Delta_{\ell,i}$ for some $\ell,i$, and there exists $p=(y,\ell)\in\hat a$,
$p'=(y',\ell')\in\hat a'$, such that
\begin{itemize}
\item[(a)] $y$ lies in the unstable manifold of a periodic point, and similarly for $y'$.
\item[(b)] $D(p,p')\neq0$.
\end{itemize}
\end{proposition}

\begin{proof}
According to Theorem~\ref{thm-D} there exist $p=(y,\ell)$, $p'=(y',\ell)$ such
that $D(p,p')\neq0$. 
Let $z_n\to y$ be a sequence of periodic points as in Proposition~\ref{prop-periodic} and let $y_n=[z_n,y]$ so $y_n\to y$.
Also by Proposition~\ref{prop-periodic}, $[y_n,y']=[z_n,y']\to[y,y']$.
We have 
\[
D(y_n,y')
=D_0(y_n,[y_n,y'])+ D_0(y',[y',y_n]) 
=D_0(y_n,[y_n,y'])+ D_0(y',[y',y]).
\]
By Lemma~\ref{lem-D0}, $D_0(y_n,[y_n,y'])\to D_0(y,[y,y'])$.

Now let $q_n=(z_n,\ell)$ and $p_n=[q_n,p]=(y_n,\ell)$.
Since $D_0((a,\ell),(b,\ell))-D_0(a,b)$ is a finite sum (with $2\ell$ terms)
of continuous functions, it follows that
$D_0(p_n,[p_n,p'])\to D_0(p,[p,p'])$ and hence
that
$D(p_n,p')\to D(p,p')$.
Hence there exists $n$ such that $D(p_n,p')\neq0$ and so we can replace $p$
by the point $p_n=(y_n,\ell)$ where $y_n$ lies in the unstable manifold of the periodic point $z_n$ while maintaining condition (b).   Similarly, we can replace $p'$ by a point $(y'_n,\ell)$ where $y'_n$ lies in the unstable manifold of a periodic point.
\end{proof}

Now we fix the points $p=(y,\ell)$, $p'=(y',\ell)$ from Proposition~\ref{prop-nonzero} and
consider the map $g:W^u(p)\to\R$ given by
\[
g(q)=D(q,p')= D_0(q,[q,p'])+D_0(p',[p',q]).
\]
Since $W^u(p)$ is naturally identified with $f^\ell W_F^u(y)$ it makes sense
to speak of smoothness of $g$.

\begin{proposition} \label{prop-g}
The one-dimensional map $g:W^u(p)\to\R$ is $C^1$.
\end{proposition}

\begin{proof}
  Let $g_1(q)=D_0(q,[q,p'])= \sum_{j=1}^\infty
  r(f^{-j}q)-\sum_{j=1}^\infty r(f^{-j}[q,p'])$.  Since $[q,p']=[p,p']$ is
  independent of $q\in W^u(p)$, the second sum consists of
  constant functions.  For the first sum, note that each
  $z\in W^u(y)$ converges in backwards time to the periodic
  orbit $y$.  Since $p=f^\ell y$ and $q=f^\ell z$, the backwards trajectory $\{f^{-j}q,\;
  j\ge1\}$ is bounded away from the singularity at $0$.  It
  follows that along this trajectory $f^{-1}$ is uniformly
  contracting and $r$ is uniformly $C^1$.  (The 
  constants are uniform in $j$ but are allowed to depend on $q$.)  Hence the series
  for $(dg_1)_q:T_qW^u(p)\to\R$ is absolutely convergent
  and $g_1$ is $C^1$.

  A similar argument applies to $g_2(q)=D_0(p',[p',q])=
  \sum_{j=1}^\infty r(f^{-j}p')-\sum_{j=1}^\infty r(f^{-j}[p',q])$.  This
  time, it is the first sum that consists of constants.  The
  second sum is like the first sum in $g_1$ with $q$
  replaced by $[p',q]$  which converges in backwards
  time to the periodic orbit $y'$.  It follows that
  the dependence of $g_2$ on $[p',q]$ is $C^1$.  But
  $z\mapsto [y',z]$ is $C^1$ by the last statement of
  Proposition~\ref{prop-localprod}
and so $q\mapsto [p',q]$ is $C^1$.  Hence $g_2$ is $C^1$
  and so $g=g_1+g_2$ is $C^1$.
\end{proof}

\begin{corollary} \label{cor-g}
There exists a nonempty open set $V\subset W^u(p)\cong W_F^u(f^\ell y)$ on which $g$
restricts to a $C^1$ diffeomorphism.
\end{corollary}

\begin{proof}
By Proposition~\ref{prop-g}, $g$ is a $C^1$ map on $W^u(p)$.
Since $g([p,p'])=0$ and $g(p)\neq 0$ by assumption, it follows 
that $g'$ is not identically zero and the result follows.
\end{proof}

\subsection{Dimension of the range of the temporal distortion function}
\label{sec-Hdim}

If necessary, we now choose a new inducing
scheme $F^*:Y^*\to Y^*$ with $Y^*\subset \bigcup_{v\in V}W^s_f(v)$
and such that the properties in
Sections~\ref{sec:lorenz-attractors-as}
and~\ref{sec:local-product-struct} remain valid.  
(For this part of the argument it suffices to take $Y^*$ connected and $F^*$ full branch.)
The definition of $D$, and hence $g$, is unchanged since this is
defined in terms of $r$ and $f$, independent of the
choice of inducing scheme.  
Let $\alpha^*$ denote the associated partition of $Y^*$ and
choose two partition elements
$a_1,a_2\in\alpha^*$.
Define the finite subsystem
$A_0=\bigcap_{n\ge0} (F^*)^{-n}(a_1\cup a_2)$.

\begin{proposition} \label{prop-dim}
For the finite subsystem $A_0$ constructed above,
the set $D(A_0\times A_0)$ has positive
lower box dimension.
\end{proposition}

\begin{proof}
  At the level of the quotient dynamics, the map
  $\bar F^*:\bar Y^*\to\bar Y^*$ is uniformly expanding.  
Moreover, $\bar F^*a_i=\bar Y^*$ for $i=1,2$ and the derivative of $\bar F^*$ is bounded on the closure of $a_1\cup a_2$.
It follows 
  (see for example~\cite[p.~203]{Takens}) that the Cantor set
  $A_0$ has positive Hausdorff dimension.  Since $g|_V$ is a
  $C^1$ diffeomorphism and $A_0\subset V$, it follows that
  $g(A_0)$ has positive Hausdorff dimension.  Hence the
  larger set $D(A_0\times A_0)$ has positive lower box dimension.
\end{proof}

\section{Fast mixing decay of correlations}
\label{sec:fast-mixing-decay}

We are now ready to complete the proof of
Theorem~\ref{thm:rapid}.  According to~\cite{Melb07,Melb09}, the
result is immediate from Proposition~\ref{prop-dim}.  
Unfortunately, the precise formulation of the result we require is not written down there.   If the roof function were bounded then we would have all the ingredients required to directly apply~\cite[Corollary~5.6]{Melb09}.
The case of unbounded roof functions is considered in~\cite[Proposition~3.6]{Melb07} and~\cite[Section~6.5]{Melb09} for semiflows and flows respectively, but
omitting the crucial ingredient provided by the dimension of the range of the temporal distortion function.

Hence, to apply the results in~\cite{Melb07,Melb09} it is necessary to recall several of the definitions and intermediate steps.  This is done for semiflows in
Subsection~\ref{sec:semiflow}.  In Section~\ref{sec:flow}, we pass from the semiflow to the flow; here it turns out to be particularly convenient to use a recent approach of~\cite{AvGoYoc},

\subsection{Fast mixing for the semiflow}
\label{sec:semiflow}

We assume that $\bar F:\bar Y\to\bar Y$ is a uniformly expanding map with partition $\alpha_0$ covered by a uniformly hyperbolic map $F:Y\to Y$ with partition $\alpha$ as in Section~\ref{sec:lorenz-attractors-as}.
We continue to suppose 
that $R:Y\to\R^+$ is a possibly unbounded roof
function, constant along stable leaves, 
that is locally Lipschitz in the symbolic metric $d_\theta$ on  $\bar Y$.
Moreover, $R$ is bounded below and has exponential tails.  

Given $v:\bar Y^R\to\R$ continuous, we define
$|v|_\theta=\sup |v(y,u)-v(y',u)|/d_\theta(y,y')$ where
the supremum is over distinct points $(y,u),\,(y',u)\in \bar Y^R$.
(Recall that $\bar Y^R$ is an identification space so observables $v:\bar Y^R\to\R$ satisfy $v(y,R(y))=v(\bar Fy,0)$.)
Define $F_{\theta}(\bar Y^R)$ to be the space
of continuous observables $v:\bar Y^R\to\R$ 
such that $\|v\|_\theta=|v|_\infty+|v|_\theta<\infty$.

Let $\partial_tv=\frac{d}{dt}\bar S_tv|_{t=0}$ denote the derivative of $v$ in the flow direction.  So $\partial_tv(y,u)=\frac{\partial}{\partial u}v(y,u)$ when $0<u<R(y)$, 
$\partial_t v(y,0)=\lim_{t\to0+}(v(y,t)-v(y,0))/t$ and
$\partial_t v(y,R(y))=\lim_{t\to0+}(v(y,R(y))-v(y,R(y)-t))/t$.
Provided that
$\partial_t v(y,R(y))= \partial_t v(\bar Fy,0)$ for all $y\in\bar Y$, 
this defines a function $\partial_tv:\bar Y^R\to\R$.
If in addition $v,\partial_tv\in F_{\theta}(\bar Y^R)$ then we write $v\in F_{\theta,1}(\bar Y^R)$.

Similarly, define the space $F_{\theta,k}(\bar Y^R)$ of observables
$v:\bar Y^R\to\R$ that are $C^k$ in the semiflow direction with derivatives
$\partial_t^jv\in F_\theta(\bar Y^R)$ for $j=0,\dots,k$.
Define $\|v\|_{\theta,k}=\sum_{j=0}^k \|\partial_t^jv\|_\theta$.

We require some further definitions from~\cite{Melb07,Melb09} based on~\cite{dolgopyat98}.
A subset $\bar A_0\subset\bar Y$ is a {\em finite subsystem} of $\bar Y$ if
$\bar A_0=\bigcap_{n\ge1}\bar F^{-n}\bar A$ where
$\bar A$ is a finite union of elements of $\alpha_0$.
Similarly, a subset $A_0\subset Y$ is a {\em finite subsystem} of $Y$ if
$A_0=\bigcap_{n\ge1}F^{-n}A$ where $A$ is a finite union of elements of $\alpha_0$.
Such a finite subsystem projects to a finite subsystem $\bar A_0\subset\bar Y$.

\begin{definition}  \label{def-approx}
For $b\in\R$ define $M_b:L^\infty(\bar Y)\to L^\infty(\bar Y)$,
$M_b v=e^{-ibR}v\circ\bar F$.
	We say that $M_b$ has {\em approximate eigenfunctions} on a subset $\bar A_0\subset \bar Y$ if there exist constants $\alpha,\beta>0$ arbitrarily large and $C\ge1$, and sequences $b_k\in\R$ with $|b_k|\to\infty$, $\varphi_k\in[0,2\pi)$, $u_k:\bar Y\to\C$ with $|u_k|\equiv1$ and $|u_k|_\theta=\sup_{y\neq y'}|u_k(y)-u_k(y')|/d_\theta(y,y')\le C|b_k|$, such that setting
	$n_k=[\beta \ln|b_k|]$,
	\[
	|(M_{b_k}^{n_k}u_k)(y)-e^{i\varphi_k}u_k(y)|\le C|b_k|^{-\alpha},
	\]
	for all $y\in\bar A_0$ and all $k\ge1$.
	\end{definition}


	\begin{theorem} \label{thm-M07} Let $\bar S_t:\bar Y^R\to\bar Y^R$
	  be a suspension semiflow 
	over a uniformly expanding map $\bar F:\bar Y\to\bar Y$,
	where the roof function $R:\bar Y\to\R^+$ has exponential tails.

	Suppose that there exists a finite subsystem $\bar A_0\subset \bar Y$ such that  there are no approximate eigenfunctions on $\bar A_0$.  Then
	the semiflow has superpolynomial decay for sufficiently
	smooth observables. That is, 
	for any $\gamma>0$, there exists $C>0$ and $k\ge1$ such that
	for all 
	observables $v\in F_{\theta,k}(\bar Y^R)$, $w\in L^\infty(\bar Y^R)$ and all $t>0$,
	\begin{align} \label{eq-mixing}
	\Big|
	\int v \; w\circ \bar S_t \, d\bar\mu -
	\int v\,d\bar\mu  \int w\, d\bar\mu\Big|
	\le 
	C\|v\|_{\theta,k} |w|_{\infty}t^{-\gamma}.
	\end{align}
	\end{theorem}

	\begin{proof}
For the quotient suspension $\bar Y^R$, we are in the situation of~\cite[Section 3]{Melb07}.
(The induced roof function $R$ is denoted by $H$ in~\cite{Melb07}.)
The exponential tail condition in~\cite[Definition~3.1]{Melb07} follows from
Proposition~\ref{prop-Rdecay} and Lemma~\ref{le:Lip-symbolic}.
	Hence, in principle, the result follows from~\cite[Lemma~3.5, Proposition~3.6]{Melb07}.   There are two caveats that need to be mentioned.

The first caveat is that 
the definition of approximate eigenfunctions in~\cite{Melb07}
is slightly weaker than in Definition~\ref{def-approx} since the constraint $|u_k|_\theta\le C|b_k|$ is not mentioned.  
However, as is easily checked, the argument in~\cite{Melb07} shows that
the failure of fast mixing implies the existence of approximate eigenfunctions that actually satisfy the stronger requirements of Definition~\ref{def-approx}. 
Only~\cite[Lemma~3.5]{Melb07} is possibly affected, and it is a consequence
of~\cite[Corollary~3.11 and Lemmas~3.12 and~3.13]{Melb07}. Of these, only~\cite[Lemma~3.12]{Melb07} is possibly affected by the change in definition.  Moreover, this lemma gives a criterion for the existence of approximate eigenfunctions (called $w$ and eventually $w_1$) and these are shown to satisfy the extra constraint $|w_1|_\theta\le C_{11}|b|$.

The second caveat is that~\cite[Proposition~3.6]{Melb07} has an additional hypothesis, namely that $\bar S_t$ is mixing.  We claim that if
$\bar S_t$ is not mixing, then there exist approximate eigenfunctions on the whole of $\bar Y$; hence this extra hypothesis is redundant.  It remains to verify the claim.  A standard characterisation of mixing for suspension
semiflows over a mixing transformation $\bar F$ is that for each $c\neq0$ the functional equation $u\circ \bar F=e^{icR}u$ has no measurable solutions
$u:\bar Y\to S^1$ where $S^1\subset\C$ is the unit circle.
Suppose that $c\neq0$ and $u:\bar Y\to S^1$ measurable satisfy such a
functional equation.  Since $|e^{icR(y)}-e^{icR(y')}|\le |c||R(y)-R(y')|$,
the exponential tail condition on $R$ certainly implies that the hypothesis
on $f=e^{icR}$ in~\cite[Theorem~1.1]{Gouezel06} is satisfied.
Hence $u$ has a version with $|u|_\theta<\infty$.
For any $\alpha,\beta>0$, we let $u_k=u^k$,
$b_k=kc$, $n_k=[\beta\ln|b_k|]$, $\varphi_k=0$.
In particular, $|u_k|_\theta\le k|u|_\theta\le C|b_k|$ with $C=|u|_\theta/|c|$.
Moreover, $M_{b_k}^{n_k}u_k\equiv e^{i\varphi_k}u_k$ so the requirements
in Definition~\ref{def-approx} are satisfied.
This completes the proof.
	\end{proof}

\begin{proposition}
\label{prop-range}
Let $A_0\subset Y$ be a finite subsystem and suppose that $D(A_0\times A_0)$
has positive lower box dimension.
Then there are no approximate eigenfunctions on $\bar A_0$.
\end{proposition}

\begin{proof}
Suppose 
that there are approximate eigenfunctions on $\bar A_0$.
The calculation in the proof of~\cite[Theorem~5.5]{Melb09} (with $\omega_k=0$)
shows that for all $\alpha>0$, there is a sequence $b_k\in\R$ with $|b_k|\to\infty$ and a constant $C>0$ such that $|e^{ib_kD(y_1,y_4)}-1|\le C|b_k|^{-\alpha}$
for all $y_1,y_4\in A_0$.
It then follows from~\cite[Corollary~5.6]{Melb09} that 
$D(A_0\times A_0)$ has lower box dimension zero.
\end{proof}

\begin{corollary}  \label{cor:rapid}
Let $G\in\mathcal{U}$ be a vector field defining a strongly dissipative
geometric Lorenz flow.
Let $\bar S_t:\bar Y^R\to\bar Y^R$ denote the corresponding suspension semiflow.
Then for all $\gamma>0$, there exists $C>0$ and $k\ge1$ such that 
for all observables $v\in F_{\theta,k}(\bar Y^R)$, $w\in L^\infty(\bar Y^R)$, and all $t>0$,
\[
\Big|
\int v \; w\circ \bar S_t \, d\bar\mu -
\int v\,d\bar\mu  \int w\, d\bar\mu\Big|
\le 
	C\|v\|_{\theta,k} |w|_{\infty}t^{-\gamma}.
\]
	\end{corollary}

\begin{proof}
	In Proposition~\ref{prop-dim}, we constructed a finite subsystem $A_0\subset Y$
such that $D(A_0\times A_0)$ has positive lower box dimension.
Hence there are no approximate eigenfunctions on $\bar A_0$ by 
Proposition~\ref{prop-range}.
The result follows from Theorem~\ref{thm-M07}.
\end{proof}

\begin{remark} \label{rmk-simplify} As mentioned in
  Remark~\ref{rmk-Rdecay}, the full strength of
  Proposition~\ref{prop-Rdecay} is not required in this
  paper.  A standard and elementary argument using the
  exponential tails for $r$ and $\tau$ implies the stretched
  exponential estimate $\mu_Y(R>n)=O(e^{-ct^{1/2}})$ which
  suffices for the methods in~\cite{Melb09}.  However, the
  analogue of Theorem~\ref{thm-M07} is not stated explicitly
  in~\cite{Melb09} so for ease of exposition we have made
  use of Proposition~\ref{prop-Rdecay}.
\end{remark}

\subsection{Fast mixing for the flow}
\label{sec:flow}

We continue to assume the structure from Subsection~\ref{sec:semiflow} and in addition that 
there is a $C^{1+\epsilon}$ global exponentially contracting stable foliation.  

In Subsection~\ref{sec:semiflow},
we defined the space $F_{\theta}(\bar Y^R)$ of observables on $\bar Y$.
To define the corresponding space
$F_{\theta}(Y^R)$ it is convenient to choose coordinates
$(y,z)$ on $Y$ where $y\in\bar Y$ and vertical lines correspond to stable leaves (recall that this can be achieved by a $C^{1+\epsilon}$ change of coordinates).
Then we define $F_\theta(Y^R)$  to be the space of continuous observables 
$v:Y^R\to\R$ respecting the identifications
$(y,z,R(y))\sim (F(y,z),0)$ and such that
$\|v\|_\theta=|v|_\infty+|v|_\theta<\infty$ where
\[
|v|_\theta=\sup_{(y,z,u)\neq(y',z',u)}\frac{|v(y,z,u)-v(y',z',u)|}{d_\theta(y,y')+|z-z'|}.
\]
Again, we define the space $F_{\theta,k}(Y^R)$ of observables
$v:Y^R\to\R$ that are $C^k$ in the flow direction with derivatives
$\partial_t^jv\in F_\theta(Y^R)$ for $j=0,\dots,k$.
Define $\|v\|_{\theta,k}=\sum_{j=0}^k \|\partial_t^jv\|_\theta$.


\begin{lemma} \label{lem-AGY}
There is a continuous family of probability measures $\{\eta_y,\,y\in\bar Y\}$ 
on $Y$ with $\supp \eta_y\subset \pi^{-1}(y)$ such that
$\int_Y v\,d\mu_Y=\int_{\bar Y} \int_{\pi^{-1}(y)} v\,d\eta_y\,d\mu_{\bar Y}(y)$
for all $v:Y\to\R$ continuous.

Moreover, there is a constant $C_1>0$ such that
if $v\in F_{\theta,k}(Y^R)$, and $\bar v:\bar Y^R\to\R$ is defined to be
$\bar v(y,u)=\int_{\pi^{-1}(y)} v(y',u)\,d\eta_y(y')$, then 
$\bar v\in F_{\theta,k}(\bar Y^R)$
and 
$\|\bar v\|_{\theta,k}\le C_1\|v\|_{\theta,k}$.
\end{lemma}

\begin{proof}
	A proof of the existence of the continuous disintegration $\mu_Y=\int_{\bar Y}\eta_y\,d\bar \mu_Y(y)$ in the first statement of the lemma can be found for instance in~\cite{BM15}.   

Suppose that $v:Y^R\to\R$ is continuous.  In particular $v(y,R(y))=v(Fy,0)$ for all $y\in Y$.   Define 
$\bar v(y,u)=\int_{\pi^{-1}(y)} v(y',u)\,d\eta_y(y')$.  As shown below,
$\bar v(y,R(y))=\bar v(\bar Fy,0)$ for all $y\in\bar Y$, so that we have a well-defined function
$\bar v:\bar Y^R\to\R$.   The estimate $\|\bar v\|_\theta\le C_1\|v\|_\theta$
follows from~\cite[Proposition~10]{BM15}. The case of general $k$ follows since 
\[
\partial_t^j\bar v(y,u)
=\int_{\pi^{-1}(y)} \partial_t^jv(y',u)\,d\eta_y(y')
=\overline{\partial_t^jv}(y,u), \enspace\text{ for all $j$}.
\]

It remains to prove that
$\bar v(y,R(y))=\bar v(\bar Fy,0)$ for $y\in\bar Y$.
Throughout, we regard $\{\eta_y,\,y\in\bar Y\}$ as a family of probability measures
on $Y$. In particular, $F_*\eta_y$ denotes the pushforward of $\eta_y$
(so $(F_*\eta_y)(E)=\eta_y(F^{-1}E)$).

Define $v_0:Y\to\R$ by setting $v_0(y)=v(y,0)$.  
Then $\bar v(y,0)=\eta_{y}(v_0)$, and so 
$\bar v(\bar Fy,0)=\eta_{\bar Fy}(v_0)$.
  Also, using that $R$ is constant along the stable foliation, we obtain
\begin{align*}
\bar v(y,R(y)) 
& =\int_Y v(y',R(y))\,d\eta_y(y') 
= \int_Y v(y',R(y'))\,d\eta_y(y') \nonumber
\\
&= \int_Y v(F y',0)\,d\eta_y(y')
=\eta_y(v_0\circ F)=(F_*\eta_y)(v_0).
\end{align*}
By continuity of $\bar v$,  it remains to show that $F_*\eta_y= \eta_{\bar Fy}$
for $\mu_{\bar Y}$-a.e.\ $y\in\bar Y$.   

We claim that
\[
\int_{\bar Y}F_*\eta_y\, d\mu_{\bar Y}(y)=
\int_{\bar Y}\eta_{\bar Fy}\,d\mu_{\bar Y}(y).
\]
By uniqueness of a family of conditional measures with respect to a given measure and measurable partition, we deduce from the claim that
$F_*\eta_y= \eta_{\bar Fy}$ 
for $\mu_{\bar Y}$-a.e.\ $y\in\bar Y$ as required.

It remains to prove the claim.  Since $\mu_{\bar Y}$ is $\bar F$-invariant,
we have on the one hand
\[
\int_{\bar Y}\eta_{\bar Fy}\,d\mu_{\bar Y}(y)=
\int_{\bar Y}\eta_{y}\,d(\bar F_*\mu_{\bar Y})(y)=
\int_{\bar Y}\eta_{y}\,d\mu_{\bar Y}(y).
\]
On the other hand, by $F$-invariance of $\mu_Y$,
\[
\int_{\bar Y}\eta_y\,d\mu_{\bar Y}(y)=\mu_Y=F_*\mu_Y
=\int_{\bar Y}F_*\eta_y\,d\mu_{\bar Y}(y),
\]
completing the proof of the claim.
\end{proof}

\begin{pfof}{Theorem~\ref{thm:rapid}}
We follow the argument in~\cite[Section~8]{AvGoYoc}.

Given $\gamma>0$, choose $C>0$ and $k\ge1$ as in Corollary~\ref{cor:rapid}.
We suppose that $v\in C^k(\R^3)$ and that $w\in C^\alpha(\R^3)$ for some $\alpha>0$.

Recall that $p:Y^R\to \R^3$ is the semiconjugacy $p(y,u)=Z_uy$.
It suffices to prove the result for observables
$v\circ p$ and $w\circ p$ at the level of the suspension flow on $Y^R$.
By Proposition~\ref{prop-p},
$v\circ p\in F_{\theta,k}(Y^R)$.

Without loss, we may suppose that $\int_{Y^R} v\circ p\,d\mu=0$.
Define the semiconjugacy $\pi^R:Y^R\to \bar Y^R$, $\pi^R(y,u)=(\pi y,u)$, so
$\bar S_t\circ \pi^R=\pi^R\circ S_t$.

Define $w_t:Y^R\to\R$ by setting
\[
w_t(y,u)=\int_{\pi^{-1}(y)} w\circ p\circ S_t(y',u)\,d\eta_{y}(y').
\]
Then $\int_{Y^R} v\circ p\;w\circ p\circ S_{2t}\,d\mu
=I_1(t)+I_2(t)$, where
\begin{align*}
I_1(t) & = \int_{Y^R} v\circ p\;w\circ p\circ S_{2t}\,d\mu-
 \int_{Y^R} v\circ p\;w_t\circ \bar S_t\circ\pi^R\,d\mu \\
 I_2(t) & = \int_{Y^R} v\circ p\;w_t\circ \bar S_t\circ\pi^R\,d\mu.
\end{align*}

Now $I_1(t)= \int_{Y^R} v\circ p\,\{(w\circ p\circ S_t- w_t\circ \pi^R)\circ S_t\}\,d\mu$, so
$|I_1(t)|\le |v|_1 |w\circ p\circ S_t- w_t\circ \pi^R|_\infty$.
Using the definitions of $\pi^R$ and $w_t$,
\begin{align*}
w\circ p\circ S_t(y,u)- w_t\circ \pi^R(y,u) & =
w\circ p\circ S_t(y,u)- w_t(\pi y,u)\\  & =
\int_{\pi^{-1}(y)}
(w\circ p\circ S_t(y,u)- w\circ p\circ S_t (y',u))\,d\eta_{\pi(y)}(y').
\end{align*}
Since $S_t$ contracts exponentially along stable manifolds, there
are constants $C_2,a>0$ (dependent on $\alpha$) such that
$|w\circ p\circ S_t(y,u)- w\circ p\circ S_t (y',u)|\le 
C_2|w|_\alpha e^{-at}$ for all $y'\in\pi^{-1}(y)$, $(y,u)\in Y^R$, and so
$|w\circ p\circ S_t- w_t\circ \pi^R|_\infty  \le C_2|w|_\alpha e^{-at}$.
Hence 
\begin{align*}
|I_1(t)| & 
\le C_2|v|_1|w|_\alpha e^{-at}.
\end{align*}

Next, define $\bar v:\bar Y^R\to \R$ by setting
$\bar v(y,u)=\int_{\pi^{-1}(y)} v\circ p(y',u)\,d\eta_{y}(y')$.
Since $\int_{Y^R}v\circ p\,d\mu=0$, it follows from Lemma~\ref{lem-AGY} that
$\int_{\bar Y^R}\bar v\,d\bar\mu=0$.  Moreover,
$I_2(t)=\int_{\bar Y^R}\bar v\;w_t\circ \bar S_t\,d\bar\mu$.
By Lemma~\ref{lem-AGY}, $\bar v\in F_{\theta,k}(\bar Y^R)$ and $\|\bar v\|_{\theta,k}\le C_1\|v\circ p\|_{\theta,k}$.
Clearly, $|w_t|_\infty\le |w|_\infty$.
Hence it follows from Corollary~\ref{cor:rapid} that 
$|I_2(t)|\le C\|\bar v\|_{\theta,k}|w_t|_\infty\, t^{-\gamma}
\le CC_1 \|v\|_{\theta,k}|w|_\infty\,t^{-\gamma}$
completing the proof.
%
	\end{pfof}


\section{ASIP for time-$1$ map of a nonuniformly expanding semiflow}
\label{sec-ASIPsemiflow}

In this section, we prove the ASIP for time-$1$ maps of a
general class of sufficiently mixing nonuniformly expanding
semiflows.

Suppose that $\bar F:\bar Y\to \bar Y$ is a Gibbs-Markov map
(uniformly expanding with bounded distortion and big images --
for notational convenience we assume full branches)
with ergodic invariant measure $\mu_{\bar Y}$.  Let
$d_\theta$ denote a symbolic metric on $\bar Y$ for some
$\theta\in(0,1)$.  Let $R:\bar Y\to\R^+$ be a possibly
unbounded roof function satisfying
\begin{itemize}
\item[(i)] $R$ is bounded below,
\item[(ii)] $\mu_{\bar Y}(R>t)=O(t^{-\beta})$ for some $\beta>1$,
\item[(iii)] $|R|_\theta=\sup_{y\neq
    y'}|R(y)-R(y')|/d_\theta(y,y')<\infty$.
\end{itemize}
(This includes the case of uniformly expanding semiflows
where $R$ is bounded.)  Define the suspension semiflow
$\bar S_t:\bar Y^R\to \bar Y^R$ with ergodic invariant measure $\bar\mu=\mu_{\bar Y}\times{\rm Leb}/\int_{\bar Y}R\,d\mu_{\bar Y}$.
Let $F_{\theta,k}(\bar Y^R)$ be the space of observables introduced in
Section~\ref{sec:semiflow}.

\begin{remark} \label{rmk-iii}
Condition (iii), although satisfied for geometric Lorenz flows, is unnecessarily restrictive.  At the end of the section, we show how this condition can be relaxed.
\end{remark}

\begin{theorem} \label{thm-ASIPsemiflow} Suppose
  that~\eqref{eq-mixing} holds with $\beta>2\sqrt2+1$.  Let $v\in
  F_{\theta,k+1}(\bar Y^R)$ be an observable with mean zero.
  Then the ASIP holds for the time-$1$ map $\bar S=\bar S_1$: passing
  to an enriched probability space, there exists a sequence
  $X_0,X_1,\ldots$ of iid normal random variables with mean
  zero and variance $\sigma^2$, such
  that
\[
\sum_{j=0}^{n-1}v\circ \bar S^j=\sum_{j=0}^{n-1}
X_j+O(n^{1/4}(\log n)^{1/2}(\log\log n)^{1/4}),\enspace a.e.
\]
The variance is given by
\[
\sigma^2=\lim_{n\to\infty}\frac1n\int(\sum_{j=0}^{n-1}v\circ \bar S^j)^2\,d\bar\mu=\sum_{n=-\infty}^\infty \int v \cdot \big(v\circ \bar S^n \big) \, d\bar\mu .
\]
The degenerate case $\sigma^2=0$ occurs if and only if 
$v=\chi\circ \bar S-\chi$ for some $\chi\in L^2$.  Moreover, for any
finite $p$ we have $\chi\in L^p$ for $k$ sufficiently large.
\end{theorem}

\begin{remark}
  (a) The CLT and functional CLT can be proved more directly under
  weaker assumptions: it suffices that $\beta>1$,
  see~\cite[Theorem~1]{Tyran-Kaminska05}.   
Also, the statements about the variance in Theorem~\ref{thm-ASIPsemiflow}
are a standard consequence of the methods there.

(b) An ASIP with weaker
error term can be proved for smaller values of $\beta$.
However, in our application to the geometric Lorenz
attractor, we can obtain any desired value of $\beta$ by
increasing the smoothness of the observable, and there is no
easy relationship between the degree of smoothness and the
size of $\beta$, so there seems little point in pursuing
this here.
\end{remark}

First, we relate the transfer operators for the semiflow
and the induced map $\bar F$.
Let $L_t$
be the transfer operator corresponding to the semiflow $\bar S_t$,
so $\int_{\bar Y^R} L_tv\,w\,d\bar\mu=\int_{\bar Y^R} v\,w\circ \bar S_t\,d\bar\mu$ for 
all $v\in L^1(\bar Y^R)$, $w\in L^\infty(\bar Y^R)$.  In particular,
$L_n$ is the transfer operator for $\bar S^n$.
Let $P$ denote the transfer operator for $\bar F:\bar Y\to
\bar Y$, so $\int_{\bar Y}Pv\,w\,d\mu_{\bar Y}
=\int_{\bar Y}v\,w\circ \bar F\,d\mu_{\bar Y}$
for all $v\in L^1(\bar Y)$, $w\in L^\infty(\bar Y)$.

\begin{proposition} \label{prop-Lt} Let $t>0$.  The transfer
  operator $L_t$ is given by a finite sum of the form
  \begin{align*}
    (L_tv)(y,u)= \sum_{j=0}^\infty (P^j\tilde v_{t,u,j})(y),
    \quad \text{with}\quad\tilde v_{t,u,j}(y)= v(y,u-t+R_j(y)).
  \end{align*}
The number of nonzero terms in the sum is bounded by
$t/\inf R+1$.
\end{proposition}

\begin{proof}
Let $\bar R=\int_Y R\,d\mu_{\bar Y}$ and write
\begin{align*}
\int_{\bar Y^R}v\,w\circ \bar S_t\,d\bar\mu & =
(1/\bar R)\int_Y\int_0^{R(y)}v(y,u)\,w\circ \bar S_t(y,u)\,du\,d\mu_{\bar Y}
\\ & =
(1/\bar R)\sum_{j=0}^\infty\int_Y
\int_{R_j(y)-t}^{R_{j+1}(y)-t}
1_{[0,R(y)]}(u)v(y,u)\,w(\bar S_t(y,u))\,du\,d\mu_{\bar Y}.
\end{align*}
For the $j$'th term to give a nonzero contribution to the
sum, it is necessary that $R_{j+1}(y)-t<R(y)$ for some $y$,
equivalently $R_j(\bar Fy)<t$, leading to the condition that
$0\le j\le t/\inf R$.  Now
\begin{align*}
  & \int_{\bar Y}\int_{R_j(y)-t}^{R_{j+1}(y)-t}
  1_{[0,R(y)]}(u)v(y,u)\,w(\bar S_t(y,u))\,du\,d\mu_{\bar Y} 
  \\ & \qquad
  = \int_{\bar Y}\int_{R_j(y)-t}^{R_{j+1}(y)-t}
  1_{[0,R(y)]}(u)v(y,u)\,w(\bar
  F^jy,u+t-R_j(y))\,du\,d\mu_{\bar Y}
  \\ & \qquad =
  \int_{\bar Y}\int_0^{R(\bar F^ky)}
  1_{[0,R(y)]}(u-t+R_j(y))v(y,u-t+R_j(y))\,w(\bar
  F^ky,u)\,du\,d\mu_{\bar Y}
  \\
  & \qquad = \int_0^\infty \int_{\bar Y} 1_{[0,R(\bar F^k(y))]}(u)
  1_{[0,R(y)]}(u-t+R_j(y))v(y,u-t+R_j(y))\,w(\bar
  F^ky,u)\,d\mu_{\bar Y}\,du
  \\
  & \qquad = \int_0^\infty \int_{\bar Y} \tilde v_{t,u,j}(y)
  \,1_{[0,R(\bar F^k(y))]}(u)w(\bar F^ky,u)\,d\mu_{\bar Y}\,du
  \\ &
  \qquad = \int_0^\infty \int_{\bar Y} (P^j\tilde v_{t,u,j})(y)
  \,1_{[0,R(y)]}(u)w(y,u)\,d\mu_{\bar Y}\,du 
  \\ & \qquad = \int_{\bar Y}
  \int_0^{R(y)} (P^j\tilde v_{t,u,j})(y)
  \,w(y,u)\,du\,d\mu_{\bar Y},
\end{align*}
as required.
\end{proof}

Theorem~\ref{thm-ASIPsemiflow} is a consequence of Cuny \&
Merlev\`ede~\cite[Theorem~3.2]{CunyMerlevedesub}.  
To apply~\cite{CunyMerlevedesub}, we are required
to check that the following three conditions hold:
\begin{align} \label{eq-CM1}
& \sum_{n=1}^\infty(\log n)^3 n^{5/2}|L_nv|_4^4<\infty, \\
\label{eq-CM2}
& \sum_{n=1}^\infty(\log n)^3 n|L_nv|_2^2<\infty, \\
\label{eq-CM3}
& \sum_{n=1}^\infty(\log n)^3 n^{-2}\Bigl(\sum_{i=1}^n\sum_{j=0}^{n-i}|L_i(vL_jv)-{\textstyle\int}_{\bar Y^R} vL_jv\,d\bar\mu|_2\Bigr)^2<\infty.
\end{align}

\begin{proposition}  \label{prop-dual}
Let $p\in[1,\infty)$.   Then there is a constant $C>0$ such that $|L_tv-\int v\,d\bar\mu|_p\le C\|v\|_{\theta,k}t^{-\beta/p}$
for all $v\in F_{\theta,k}(\bar Y^R)$.
\end{proposition}

\begin{proof} Following~\cite{MelTor02}, we set
  $w=\operatorname{sgn}  L_tv$ in~\eqref{eq-mixing} to
  obtain $|L_tv|_1\le C \|v\|_{\theta,k}t^{-\beta}$.  Since
  $v\in L^\infty$ and $|L_tv|_\infty\le |v|_\infty$, we
  obtain $|L_tv|_p^p\le |v|_\infty^{p-1}|L_tv|_1\le
  C\|v\|_{\theta,k}^pt^{-\beta}$.
\end{proof}

\begin{lemma} \label{lem-power}
There exists a constant $C_1$ (depending on $k$ and $\theta$)
such that
\begin{align} \label{eq-power}
\|vL_t v\|_{\theta,k}\le C_1(t+1)(|R|_\theta+1)\|v\|_{\theta,k+1}^2,
\end{align} 
for all $t\ge0$.
\end{lemma}

\begin{proof}
  It is clear that $\|vL_tv\|_{\theta,k}\le
  C\|v\|_{\theta,k}\|L_tv\|_{\theta,k}$ where $C$ depends
  only on $k$.  It remains to estimate
  $\|L_tv\|_{\theta,k}$.  By Proposition~\ref{prop-Lt}, it
  suffices to estimate $\|P^j\tilde v_{t,u,j}\|_{\theta,k}$
  uniformly in $j$ and $t$, since there are at most $t/\inf
  R +1$ elements in the sum.

  Note also that $\partial_t P^j\tilde v
  =P^j\widetilde{(\partial_tv)}$.  Hence it suffices to
  prove that $\|P^j\tilde v_{t,u,j}\|_{\theta,0}\le
  C(|R|_\theta+1) \|v\|_{\theta,1}$ uniformly in $j$ and
  $t$.

For general reasons, $|P^j\tilde v_{t,u,j}|_\infty\le
|\tilde v_{t,u,j}|_\infty =|v|_\infty$.  Next we recall that
$(Pv)(y)=\sum_{a\in\alpha}e^{p(y_a)}v(y_a)$ where $\alpha$
is the underlying partition, $y_a$ is the unique preimage
$\bar F^{-1}y$ lying in $a$ (this is where we assume full
branches; otherwise there may be no preimage and the term is
simply omitted) and $p$ is the potential.  Iterating, we
obtain
\[
(P^j\tilde v_{t,u,j})(y)=\sum_{a\in\alpha_j}e^{p_j(y_a)}v(y_a,u-t+R_j(y_a))
\]
where $\alpha_j$ is the partition of $j$-cylinders and
$p_j=\sum_{i=0}^{j-1}p\circ \bar F^i$,
$R_j=\sum_{i=0}^{j-1}R\circ \bar F^i$.  (Again $y_a$ denotes
the unique preimage $\bar F^{-j}y$ lying in $a$.)  We recall
the standard estimate for Gibbs-Markov expanding maps: there
is a constant $C>0$ such that
\begin{align} \label{eq-GMest} |e^{p_n(y)}-e^{p_n(y')}|\le
  Ce^{p_n(y)}d_\theta(\bar F^ny,\bar F_ny'), \quad\text{for all
    $y,y'\in a$, $a\in\alpha_n$, $n\ge1$.}
\end{align}
Also, an easy calculation shows that (see the proof of
Lemma~\ref{le:Lip-symbolic})
\begin{align} \label{eq-Rest} |R_n(y)-R_n(y')|\le
  (1-\theta)^{-1}|R|_\theta d_\theta(\bar F^ny,\bar
  F^ny'),\quad\text{ for all $y,y'\in a$, $a\in\alpha_n$,
    $n\ge1$.}
\end{align}

Let $(y,u),(y',u')\in \bar Y^R$.  We suppose without loss
that $u'\le u$.  Then
\begin{align} \label{eq-IV}
(P^j\tilde v_{t,u,j})(y)-
(P^j\tilde v_{t,u',j})(y')=I+II+III+IV,
\end{align}
where
\begin{align*}
I
& =
\sum_{a\in\alpha_j}(e^{p_j(y_a)}-e^{p_j(y'_a)})v(y_a,u-t+R_j(y_a)), 
\\
II 
& = 
\sum_{a\in\alpha_j}e^{p_j(y'_a)}\cdot[v(y_a,u-t+R_j(y_a))-v(y'_a,u-t+R_j(y_a))], \\
III 
& = 
\sum_{a\in\alpha_j}e^{p_j(y'_a)}\cdot[v(y_a',u-t+R_j(y_a))-v(y'_a,u-t+R_j(y'_a))], 
\\
IV 
& = 
\sum_{a\in\alpha_j}e^{p_j(y'_a)}\cdot[v(y'_a,u-t+R_j(y'_a))-v(y'_a,u'-t+R_j(y'_a))].
\end{align*}
Using~\eqref{eq-GMest} and~\eqref{eq-Rest}, $|I|\le
C\sum_{a\in\alpha_j}e^{p_j(y_a)}|v|_\infty=C|v|_\infty$, and
\[
|III|\le
\sum_{a\in\alpha_j}e^{p_j(y'_a)}|\partial_t v|_\infty|R_j(y_a)-R_j(y_a')|
\le(1-\theta)^{-1}|\partial_t v|_\infty|R|_\theta d_\theta(y,y').
\]
Also $|II|\le 
\sum_{a\in\alpha_j}e^{p_j(y'_a)}|v|_\theta d_\theta(y_a,y'_a)=
\theta^j|v|_\theta d_\theta(y_a,y'_a)$ and
$|IV|\le |\partial_tv|_\infty |u-u'|$.
Hence \newline
$\|P^j\tilde v_{t,u,j}\|_{\theta,0}\le C(|R|_\theta+1)\|v\|_{\theta,1}$
as required.
\end{proof}

\begin{corollary} \label{cor-CM3}
There exists a constant $C$ (depending on $k$, $\theta$ and $|R|_\theta$)
such that
\[
\sum_{i=1}^n\sum_{j=0}^{n-i}|L_i(vL_jv)-{\textstyle\int}_{\bar Y^R} vL_jv\,d\bar\mu|_2
\le C(n^{\sqrt2+1-\beta/2}+1)\|v\|_{\theta,k+1}^2,
\]
for all mean zero $v\in F_{\theta,k}(\bar Y^R)$.
\end{corollary}

\begin{proof}
On one hand, by Proposition~\ref{prop-dual} and Lemma~\ref{lem-power},
\[
|L_i(vL_jv)-{\textstyle\int}_{\bar Y^R}vL_jv\,d\bar\mu|_2\le C
i^{-\beta/2}\|vL_jv\|_{\theta,k}\le C'
i^{-\beta/2}(j+1)\|v\|_{\theta,k+1}^2.
\]
On the other hand,
\begin{align*}
|L_i(vL_jv)-{\textstyle\int}_{\bar Y^R}vL_jv\,d\bar\mu|_2 & \le
|vL_jv-{\textstyle\int}_{\bar Y^R}vL_jv\,d\bar\mu|_2 \\ & \le 
2|vL_jv|_2 \le 2|v|_\infty|L_jv|_2\le C (j+1)^{-\beta/2}\|v\|_{\theta,k}^2.
\end{align*}
Using the first estimate,
\begin{align*}
\sum_{i=1}^n\sum_{j=0}^{i^{1/\sqrt2}}
|L_i(vL_jv)-{\textstyle\int}_{\bar Y^R}vL_jv\,d\bar\mu|_2\le C\|v\|_{\theta,k+1}^2
\sum_{i=1}^n\sum_{j=0}^{i^{1/\sqrt 2}}(j+1)i^{-\beta/2}\le C'\|v\|_{\theta,k+1}^2(n^{\sqrt 2-\beta/2+1}+1).
\end{align*}
Using the second estimate,
\begin{align*}
\sum_{j=0}^n\sum_{i=1}^{j^{\sqrt2}}
|L_i(vL_jv)-{\textstyle\int}_{\bar Y^R}vL_jv\,d\bar\mu|_2\le C\|v\|_{\theta,k+1}^2
\sum_{j=0}^n\sum_{i=1}^{j^{\sqrt2}}(j+1)^{-\beta/2}\le C'\|v\|_{\theta,k+1}^2(n^{\sqrt2-\beta/2+1}+1).
\end{align*}
Combining these gives the required estimate.
\end{proof}

Now we can complete the proof of Theorem~\ref{thm-ASIPsemiflow}.

\begin{proof}[Proof of Theorem~\ref{thm-ASIPsemiflow}]
By Proposition~\ref{prop-dual}, $|L_nv|_p^p\le C^p \|v\|_{\theta,k}^p n^{-\beta}$.
Hence
conditions~\eqref{eq-CM1} and~\eqref{eq-CM2} are satisfied for $\beta>\frac72$.
By Corollary~\ref{cor-CM3}, condition~\eqref{eq-CM3} is satisfied for $\beta>2\sqrt2+1$.
\end{proof}

Finally, as promised in Remark~\ref{rmk-iii}, we show how condition (iii) can be relaxed.
Indeed it suffices that
\begin{itemize}
\item[(iii$_1$)] $\sum_{a\in\alpha}\mu_{\bar Y}(a)\Lip_aR<\infty$,
\end{itemize}
where $\Lip_Ag=\sup_{x,y\in A,\,x\neq y}|g(x)-g(y)|/d_\theta(x,y)$ for
$g:\bar Y\to\R$, $A\subset\bar Y$.

We begin by recalling a standard estimate for Gibbs-Markov maps which we did not explicitly make use of earlier: there is a constant $C>0$ such that $|1_a e^{p_j}|_\infty\le C\mu_{\bar Y}(a)$ for all $a\in\alpha_j$, $j\ge1$.

Note that condition~(iii) was only used in the estimate of term $III$ in the proof of Lemma~\ref{lem-power}.   
But alternatively, we compute that
$|III|\le |\partial_t v|_\infty\sum_{i=0}^{j-1} A_{i,j}$
where
\begin{align*}
A_{i,j}  = \sum_{a\in\alpha_j}e^{p_j(y_a')}|R\circ \bar F^i(y_a)-R\circ \bar F^i(y_a')|
& \le  C\sum_{a\in\alpha_j}\mu_{\bar Y}(a)\Lip_{\bar F^ia}R d_\theta(\bar F^iy_a,\bar F^iy_a')
\\ & =  C\theta^{j-i}d_\theta(y,y')\sum_{a\in\alpha_j}\mu_{\bar Y}(a)\Lip_{\bar F^ia}R.
\end{align*}
Now
\begin{align*}
\sum_{a\in\alpha_j}\mu_{\bar Y}(a)\Lip_{\bar F^ia}R
& =   \sum_{b\in\alpha_{j-i}}\sum_{a\in\alpha_j:\bar F^ia=b}\mu_{\bar Y}(a)\Lip_bR
=   \sum_{b\in\alpha_{j-i}}\mu_{\bar Y}(\bar F^{-i}b)\Lip_bR
=   \sum_{b\in\alpha_{j-i}}\mu_{\bar Y}(b)\Lip_bR \\
& =   \sum_{c\in\alpha}\sum_{b\in\alpha_{j-i}:b\subset c}\mu_{\bar Y}(b)\Lip_bR 
 \le    \sum_{c\in\alpha}\sum_{b\in\alpha_{j-i}:b\subset c}\mu_{\bar Y}(b)\Lip_cR 
 =    \sum_{c\in\alpha}\mu_{\bar Y}(c)\Lip_cR .
\end{align*}
Hence $A_{i,j}\le C\theta^{j-i}d_\theta(y,y')\sum_{a\in\alpha}\mu_{\bar Y}(a)\Lip_aR$ and
\[
|III|\le C\theta(1-\theta)^{-1}|\partial_tv|_\infty \sum_{a\in\alpha}\mu_{\bar Y}(a)\Lip_aR\, d_\theta(y,y').
\]
Hence, assuming condition~(iii$_1$) instead of~(iii), we obtain an estimate
analogous to the one in Lemma~\ref{lem-power} with the
conclusion~\eqref{eq-power} replaced by the estimate
\begin{align*} 
\|vL_t v\|_{\theta,k}\le C_1(t+1)\Bigl(
\sum_{a\in\alpha}\mu_{\bar Y}(a)\Lip_aR+1\Bigr)\|v\|_{\theta,k+1}^2,
\end{align*} 
for all $t\ge0$.

\section{Nondegeneracy in the CLT and ASIP}
\label{sec-sigma}

In this section, we prove that the degenerate case $\sigma^2=0$ in
Theorem~\ref{thm-ASIPsemiflow} is of infinite codimension.
Suppose as in Theorem~\ref{thm-ASIPsemiflow} that
$v=\chi\circ \bar S-\chi$ for some $\chi\in L^2$.
Following~\cite[Proposition~2]{MelTor02}, we define
$\psi_t=\int_0^t v\circ \bar S_s\,ds$ and $h=\int_0^1 \chi\circ \bar S_s\,ds$.
Then $\psi_t$ is a continuous cocycle for the semiflow $\bar S_t$;
that is $\psi_{t+s}=\psi_s\circ \bar S_t+\psi_t$ for all $s,t\ge0$.
Moreover, $h\in L^2$ and
$\psi_t=h\circ \bar S_t-h$, so $\psi_t$ is an $L^2$ coboundary.

In the Axiom~A setting of~\cite{MelTor02} it now follows from
a Liv\v{s}ic regularity theorem of~\cite{Walkden00} that
$h$ has a H\"older version.  
Hence if $q$ is a periodic point of period $T$ for $\bar S_t$, then 
$\int_0^T v(\bar S_tq)\,dt=\psi_T(q)=h(\bar S_Tq)-h(q)=0$.
Since $\bar S_t$ has infinitely many periodic orbits, this places infinitely many
restrictions on $v$.

In the nonuniformly expanding case, the situation is similar once
we have a Liv\v{s}ic regularity theorem for nonuniformly expanding semiflows.  
As we now show, this is a straightforward combination of results of~\cite{Gouezel06} for Gibbs-Markov maps and~\cite{Walkden00} for uniformly expanding semiflows.

In the remainder of this section, we suppose as in
Section~\ref{sec-ASIPsemiflow}
that $\bar F:\bar Y\to\bar Y$ is a full-branch Gibbs-Markov map
and that  $R:\bar Y\to\R^+$ is a roof function satisfying
conditions (i), (ii) and (iii$_1$).
(The full-branch condition is relaxed in Remark~\ref{rmk:livsic}.)

Given a continuous cocycle $\psi_t$ on $\bar Y^R$, we define 
$I_\psi:\bar Y\to\R$
\[
I_\psi(y)=\psi_{R(y)}(y,0).
\]

\begin{lemma}
\label{lem:livsic}
    Let $\psi$ be a cocycle for $\bar S_t$ such that  $I_\psi:X\to G$
      satisfies
      \begin{align*}
        \sum_{a\in\alpha}\mu_{\bar Y}(a)\Lip_a I_\psi<\infty.
      \end{align*}
      Suppose that there exists $h:\bar Y^R\to \R$ measurable such that
      for all $t\ge0$: $\psi_t=h\circ \bar S_t-h$ a.e.
      Then $h$ has a version that is continuous.
  \end{lemma}

\begin{proof}
We begin by following the proof of~\cite[Theorem~3.3]{Walkden00}.
Note that the set of zero measure where
$\psi_t=h\circ \bar S_t-h$ fails for each $t\ge0$ can be made
independent of $t$; see e.g. \cite[p. 13]{CoFoSi82}.
Hence for almost every $(y,u)\in \bar Y^R$,
\begin{align}\label{eq:cobd}
  \psi_{R(y)}(y,u)=h(\bar S_{R(y)}(y,u))-h(y,u).
\end{align}
Then
$\Psi(y,u)=\psi_{R(y)}(y,u)-h(\bar S_{R(y)}(y,u))+h(y,u)=0$,
$\bar\mu$-a.e. and so by Fubini's Theorem there exists
$0<u_0<\inf R$ such that $\Psi(y,u_0)=0$ for $\mu_{\bar Y}$-a.e $y\in
\bar Y$.
Since $\psi$ is a cocycle,
\begin{align*}
  \psi_{R(y)}(y,u_0)
  &=
  \psi_{R(y)}(\bar S_{u_0}(y,0))
  =
  \psi_{R(y)+u_0}(y,0)-\psi_{u_0}(y,0)
  \\
  &=
  \psi_{u_0}(\bar S_{R(y)}(y,0))+\psi_{R(y)}(y,0)-\psi_{u_0}(y,0)
  \\
  &=
  \psi_{u_0}(\bar S_{R(y)}(y,0))+I_\psi(y)-\psi_{u_0}(y,0).
\end{align*}
Substituting in (\ref{eq:cobd}) with $u=u_0$ and using
$\bar S_{R(y)}(y,u_0)=(\bar Fy,u_0)$ we obtain
\begin{align*}
  I_\psi(y)
  =
  h(\bar Fy,u_0)-\psi_{u_0}(\bar Fy,0) -  (h(y,u_0)-\psi_{u_0}(y,0))
  =
  g(\bar Fy)-g(y),
\end{align*}
where $g(y)=h(y,u_0)-\psi_{u_0}(y,0)$ is measurable.  

We have shown that $I_\psi=g\circ\bar F-g$ satisfies the hypotheses of~\cite[Theorem 1.1]{Gouezel06}.  It follows that $g$ has a version that
      is continuous (even Lipschitz) on $\bar Y$.

Let us now define
\begin{align*}
  \tilde h(y,u)=\psi_u(y,0)+ g(y).
\end{align*}
As in the proof of~\cite[Theorem 3.3]{Walkden00},
it follows from the definitions
that $\tilde h$ is a well-defined function on
$\bar Y^R$ (ie $\tilde h(y,R(y))=\tilde h(\bar F,0)$) and that
$\tilde h$  is a version of $h$.
Since $\psi_u$ and $g$ are continuous, it follows that $\tilde h$ is continuous as required.
\end{proof}

\begin{corollary} \label{cor:livsic}
Suppose that $v\in F_{\theta,0}(\bar Y^R)$ satisfies $v=\chi\circ \bar S-\chi$
for some $\chi\in L^2$.
Then 
$\int_0^T v(\bar S_tq)\,dt=0$ for all periodic points $q\in \bar Y^R$ of period $T$.
\end{corollary}

\begin{proof}
Define $\psi_t=\int_0^t v\circ \bar S_s\,ds$ and $h=\int_0^1 \chi\circ \bar S_s\,ds$,
so $\psi_t=h\circ \bar S_t-h$.
We claim that $I_\psi$ satisfies the assumption
in Lemma~\ref{lem:livsic}.
Then $h$ has  a continuous version.
Hence it follows as in the Axiom~A setting that
$\int_0^T v(\bar S_tq)\,dt=0$ for all periodic points $q\in \bar Y^R$ of period $T$.

It remains to verify the claim.
For $y\in \bar Y$, we compute that
$I_\psi(y)=\psi_{R(y)}(y,0)=\int_0^{R(y)}v\circ \bar S_s(y,0)ds=
\int_0^{R(y)} v(y,u)\,du$.   For $a\in\alpha$, $x,y\in a$, and taking $R(y)\ge R(x)$,
\begin{align*}
  & |I_\psi(x)-I_\psi(y)| = \Bigl|
    \int_0^{R(x)}\big(v(x,u)-v(y,u)\big)\,du+\int_{R(x)}^{R(y)}v(y,u)\,du
    \Bigr|
    \\
    & \qquad\qquad
\le
    R(x)|v|_\theta d_\theta(x,y)+|R(y)-R(x)||v|_\infty
    \le
    (\Sup_aR+\Lip_aR)\|v\|_\theta d_\theta(x,y).
\end{align*}
Hence $\Lip_a I_\psi\le (\Sup_aR+\Lip_aR)\|v\|_\theta$
and so
\begin{align*}
  \sum_{a\in\alpha}\mu_{\bar Y}(a)\Lip_aI_\psi\le \|v\|_\theta\sum_{a\in\alpha}\mu_{\bar Y}(a)(\Sup_aR+\Lip_aR)
  <\infty,
\end{align*}
as required.
\end{proof}

\begin{remark} \label{rmk:livsic}
The condition that
$\bar F:\bar Y\to \bar Y$ is full-branch can be relaxed as in~\cite{Gouezel06}.
The function $g:\bar Y\to\R$ constructed in the proof of Lemma~\ref{lem:livsic}
will no longer be continuous 
in general, but it is continuous on each partition element
of the 
partition $\alpha_*$ generated by the images $\bar Fa$ of the elements of $\alpha$.  We conclude that $h$ has a version that is continuous on
$\{(y,u)\in a_*\times[0,\infty):u\le R(y)\}$ for each $a_*\in\alpha_*$.
Hence in Corollary~\ref{cor:livsic} we obtain that 
$\int_0^T v(\bar S_tq)\,dt=0$ for periodic points $q\in \bar Y^R$ of period $T$
such that the orbit of $q$ intersects one of the partition elements $a_*$.
\end{remark}

\section{CLT and ASIP for the time-$1$ map of geometric
  Lorenz flows}
\label{sec-ASIPflow}

In this section we prove Theorem~\ref{thm:ASIP} (and as a
consequence Theorem~\ref{thm:CLT}), by reducing from the
geometric Lorenz flow to the quotient flow, enabling the
application of Theorem~\ref{thm-ASIPsemiflow}.

To achieve this reduction, we modify the argument
in~\cite[Appendix~A]{MelTor02} which deals with the Axiom~A
case and bounded roof function.  We note that the argument
in~\cite{MelTor02} is unnecessarily complicated, since
having reduced without loss to the situation where $r$
depends only on future coordinates, the quantity $\Delta$
in~\cite[Proposition~5]{MelTor02} is identically zero.

On the other hand, the situation for geometric Lorenz
attractors is made complicated since (a) there is no
convenient metric on the symbolic flow $Y^R$, and (b) the
roof function is unbounded.  To deal with (a), we reduce
directly to $\bar Y^R$.  For (b), we make crucial use of
Lemma~\ref{le:Lip-symbolic} and Proposition~\ref{prop-p}.

\begin{theorem} \label{thm-sinai}
Let $v:\R^3\to\R$ be a $C^{k+1}$ observable of mean zero.   
There exists $\hat v,\hat\chi:Y^R\to\R$ continuous and bounded such that
\begin{itemize}
\item[(i)] $v\circ p=\hat v+\hat\chi-\hat\chi\circ S$,
\item[(ii)] $\hat v$ depends only on future coordinates and
  hence projects to a mean zero observable $\bar v:\bar
  Y^R\to\R$,
\item[(iii)] $\bar v\in F_{\theta',k}(\bar Y^R)$ for some
  $\theta'\in(0,1)$.
\end{itemize}
\end{theorem}

We recall that 
$p:Y^R\to\R^3$ denotes the measure-preserving semiconjugacy between the flows
$S_t:Y^R\to Y^R$ and $Z_t:\R^3\to\R^3$.
Recall also that the projection $\pi:X\to \bar X$ along
stable manifolds restricts to a projection
$\pi:Y\to\bar Y$.  Also, we defined $\pi(y,u)=(\pi
y,u)$ provided that $u\in[0,R(y))$.  This induces a
measure-preserving semiconjugacy
$\pi:Y^R\to\bar Y^R$ between the flow $S_t:Y^R\to Y^R$ and the semiflow 
$\bar S_t:\bar Y^R\to \bar Y^R$.

We now show how Theorem~\ref{thm:ASIP} follows from Theorem~\ref{thm-sinai}.

\begin{proof}[Proof of Theorem~\ref{thm:ASIP}]
  Let $v:\R^3\to\R$ be a $C^{k+2}$ mean zero observable.  
By Theorem~\ref{thm-sinai}, 
  $\sum_{j=0}^{n-1}v\circ Z^j\circ p=
  \sum_{j=0}^{n-1}\hat v\circ S^j+\hat\chi-\hat\chi\circ S^n=
  \sum_{j=0}^{n-1}\bar v\circ S^j\circ \pi+O(1)$ uniformly on
  $Y^R$, where $\bar v\in F_{\theta',k+1}(\bar Y^R)$.
Hence
the ASIP for $v$ is equivalent to the ASIP for $\bar v$ and follows from
Theorem~\ref{thm-ASIPsemiflow}.

It remains to verify the statement about the degenerate case
$\sigma^2=0$ in Theorem~\ref{thm:CLT}.
By Theorem~\ref{thm-ASIPsemiflow}, $\bar v=\chi\circ S-\chi$ for some
$\chi\in L^p(\bar Y^R)$.
Working still on $\bar Y^R$,
define $\psi_t=\int_0^t \bar v\circ S_s\,ds$, 
$h=\int_0^1 \chi\circ S_s\,ds$, so $\psi_t=h\circ \bar S_t-h$.
By Lemma~\ref{lem:livsic}, $h$ has a continuous version.

Lifting to $Y^R$, we have that
$\int_0^t \hat v\circ S_s\,ds=\hat h\circ S_t-\hat h$ where
$\hat h=h\circ\pi$ is continuous.
Hence using Theorem~\ref{thm-sinai}(i),
\begin{align} \label{eq-h}
\int_0^t v\circ Z_s\circ p\,ds=\hat h\circ p\circ S_t-\hat h\circ p+
\int_0^t\tilde\chi\circ S_s\,ds-
\int_0^t\tilde\chi\circ S_{s+1}\,ds
=\tilde h\circ S_t-\tilde h,
\end{align}
where $\tilde h=\hat h\circ p-\int_0^1\tilde\chi\circ S_s\,ds$ is continuous.

Now suppose that $q$ is a periodic point of period $T_1$ for the geometric Lorenz flow $Z_t$.
Then  $q=Z_{u_0}(y_0)=p(y_0,u_0)$ for some $y_0\in Y$, $u_0\in[0,R(y_0)]$.
Since $R$ is not necessarily the first return time to $Y$ it need not be the 
case that $(y_0,u_0)$ has period $T_1$ under $S_t$.  However, certainly there exists $T>0$, an integer multiple of $T_1$, 
such that $S_T(y_0,u_0)=(y_0,u_0)$.  By~\eqref{eq-h},
$\int_0^Tv(Z_tq)\,dt=
\tilde h(S_T(y_0,u_0))-\tilde h(y_0,u_0)=0$
as required.
\end{proof}

In the remainder of this section, we prove Theorem~\ref{thm-sinai}.

Recall that the projection $p:Y^R\to \R^3$ is given
by $p(y,u)=Z_u(y)$.  Note that if $(y,u)\in Y^R$ then
$\pi Z_uy=Z_u\pi y$.  However, we caution that for
general $t>0$, $x\in \R^3$ it is {\em not} the case that
$\pi Z_tx$ and $Z_t\pi x$ coincide.

\begin{proof}[Proof of Theorem~\ref{thm-sinai}]
Define $\hat\chi:Y^R\to\R$ by setting
\[
\hat\chi(y,u)=\sum_{n=0}^\infty\{v(Z_nZ_uy)-v(Z_n\pi Z_uy)\}.
\]

It follows from exponential contraction along stable manifolds that
there are constants $C,a>0$ such that
\begin{align} \label{eq-chi_n} \nonumber
|v(Z_n(Z_uy)-v(Z_n(\pi Z_uy))| & \le |Dv|_\infty|Z_n(Z_uy)-Z_n(\pi Z_uy)|
\\ & \le C|Dv|_\infty e^{-an}|Z_uy-\pi Z_uy|\le C'|Dv|_\infty e^{-an},
\end{align}
so that $\hat\chi$ is continuous and bounded.

Define $\hat v:Y^R\to\R$ by setting
\begin{align*}
\hat v & =v\circ p+\hat\chi\circ S-\hat\chi,
\end{align*}
so that (i) is satisfied by definition.
Also
\begin{align*}
\hat v(y,u) & = v(\pi Z_u y)+\sum_{n=0}^\infty
\{v(Z_n(Z\pi Z_uy))-v(Z_n(\pi Z_{u+1}y))\},
\end{align*}
and so (ii) is satisfied.

When proving (iii), we note that the formula for $\partial_t^j\hat v$ is identical to that for $\hat v$ with $v$ replaced by $\partial_t^jv$ throughout.
Hence it suffices to consider the case $k=0$ and to prove
that $\bar v\in F_{\theta',0}(\bar Y^R)$ for $v\in C^1(\R^3)$.

By the triangle inequality it suffices to show that
\begin{align} \label{eq-barv_u}
|\bar v(y,u)-\bar v(y,u')| & \le C|Dv|_\infty |u-u'|, \\
\label{eq-barv_y}
|\bar v(y,u)-\bar v(y',u)| & \le C|Dv|_\infty(1+|R|_\theta) d_{\theta'}(y,y').
\end{align}

First we prove~\eqref{eq-barv_u}.
The $n$'th term of $\tilde v$ is given by
\[
w_n(u)=v(Z_n(Z\pi Z_uy))- v(Z_n(\pi Z_{u+1}y)).
\]
We have $|w_n(u)-w_n(u')|\le |Dw_n|_\infty |u-u'|$.
But 
\begin{align*}
|Dw_n(u)| & \le |Dv|_\infty |Z_n(Z\pi Z_uy)- Z_n(\pi Z_{u+1}y)|
\\ & \le C|Dv|_\infty e^{-an}|Z\pi Z_uy- \pi Z_{u+1}y|
\le C'|Dv|_\infty e^{-an}.
\end{align*}
Hence 
\[
|w_n(u)-w_n(u')|\le C''|Dv|_\infty e^{-an}|u-u'|,
\]
and~\eqref{eq-barv_u} follows.

It remains to prove~\eqref{eq-barv_y}.
For the initial term in the formula for $\hat v$, we note that
\begin{align} \label{eq-initial} \nonumber
|v(\pi Z_uy)-v(\pi Z_u y')| & \le |Dv|_\infty |\pi Z_uy-\pi Z_uy'|
=|Dv|_\infty |p(\pi y,u)-p(\pi y',u)| \\ &
 \le C|Dv|_\infty d_\theta(y,y'),
\end{align}
by Proposition~\ref{prop-p}.

Let $N\ge1$.
We write the remainder of $\hat v(y,u)-\hat v(y',u)$ as 
\[
A(Z\pi Z_uy,Z\pi Z_uy')+ A(\pi Z_{u+1}y,\pi Z_{u+1}y')+B(y)+B(y'),
\]
where
\begin{align*}
& A(x,x')  =\sum_{n=0}^{N-1} \{v(Z_nx)-v(Z_nx')\}, \\
& B(y)  =\sum_{n=N}^\infty \{v(Z_n(Z\pi Z_uy))-v(Z_n(\pi Z_{u+1}y))\}. 
\end{align*}

We again use exponential contraction along stable directions as in~\eqref{eq-chi_n} to show that
\begin{align} \label{eq-B}
|B(y)|,\,|B(y')|\le C|Dv|_\infty e^{-aN}.
\end{align}

Let $j=j(y,t)$ be the lap number for $y\in Y$ under $Z_t$, 
so $t\in[R_j(y),R_{j+1}(y))$ and
$Z_t(y,0)=p(F^jy,t-R_j(y))$.
Then the $n$'th term of $A(Z\pi Z_uy,Z\pi Z_u y')=A(Z_{u+1}\pi y,Z_{u+1}\pi y')$ has the form
\[
a_n=v\circ p(F^j\pi y,n+u+1-R_{j}(y))-v\circ p(F^{j'}\pi y',n+u+1-R_{j'}(y')),
\]
where 
\begin{align} \label{eq-lap}
j=j(y,n+u+1),\quad j'=j(y',n+u+1).
\end{align}
Note that $j,\,j'\le (n+1)/\inf R\le N/\inf R$.

Let $q=[1/\bar r]+2$.
Suppose that $s(y,y')= qN$.
Choose $N$ so large that $(1-\theta)^{-1}|R|_\theta\theta^N<\inf R$.
Then 
\begin{align*}
|R_j(y)-R_j(y')| & \le (1-\theta)^{-1}|R|_\theta\theta^{-j+1}\theta^{Nq}
\le (1-\theta)^{-1}|R|_\theta\theta^{N(q-1/\inf R)}
 \\ & \le (1-\theta)^{-1}|R|_\theta\theta^N<\inf R,
\end{align*}
for all $0\le j\le [N/\inf R]$.  Hence for this range of $j$, the intervals
$[R_j(y),R_{j+1}(y)]$ and
$[R_j(y'),R_{j+1}(y')]$ 
almost coincide (the initial points are within distance $\inf R$, as are the
final points).  It follows 
that the lap numbers $j$ and $j'$ 
in~\eqref{eq-lap} satisfy
$|j-j'|\le 1$ for all $0\le n\le N$.
The estimation of the terms in $A(Z\pi Z_uy,Z\pi Z_uy')$ now splits into three cases.

When $j=j'$, we obtain the term
\[
a_n=v\circ p(F^j\pi y,n+u+1-R_{j}(y))-v\circ p(F^j\pi y',n+u+1-R_{j}(y')).
\]
Hence by Proposition~\ref{prop-p},
\begin{align} \label{eq-jj'} \nonumber
 |a_n| & \le C|Dv|_\infty \{\theta^{s(F^j\pi y,F^j\pi y')}|+|R_j(y)-R_j(y')|\}
 \\ & \le C'|Dv|_\infty(1+|R|_\theta) \theta^{s(y,y')-j} 
 \le C'|Dv|_\infty(1+|R|_\theta) \theta^{qN-n/\inf R}.
\end{align}
If $j'=j+1$, then
\begin{align*}
a_n = & v\circ p(F^j\pi y,n+u+1-R_{j}(y))-
v\circ p(F^j\pi y,R(F^jy)) \\ & +
v\circ p(F^{j+1}\pi y,0)-
v\circ p(F^{j+1}\pi y',n+u+1-R_{j+1}(y')),
\end{align*}
so that
\begin{align*}
|a_n|  & \le   C|Dv|_\infty\{R_{j+1}(y)-n-u-1\} \\ &
\quad + C|Dv|_\infty\{\theta^{s(F^{j+1}y,F^{j+1}y')}+n+u+1-R_{j+1}(y')\} \\ 
& =
C|Dv|_\infty\{\theta^{s(F^{j+1}y,F^{j+1}y')}+R_{j+1}(y)-R_{j+1}(y')\},
\end{align*}
yielding the same estimate as in~\eqref{eq-jj'}.  
Similarly for the case $j'=j-1$.
Hence in all three cases, we obtain the estimate~\eqref{eq-jj'}.
Summing over $n$, we obtain that 
\begin{align} \label{eq-A1}
|A(Z\pi Z_uy,Z\pi Z_uy')| \le C|Dv|_\infty(1+|R|_\theta)\theta^{(q-1/\inf R)N}    \le C|Dv|_\infty(1+|R|_\theta)\theta^N.    
\end{align}

To deal with the $n$'th term
$A(\pi Z_{u+1}y,\pi Z_{u+1}y')$ we need to introduce four lap numbers.
First let
$j_1\le 1/\inf R$ be the lap number corresponding to $Z_{u+1}y$, so 
\[
\pi Z_{u+1}y= \pi Z_{u+1-R_{j_1}(y)}(F^{j_1}y)= Z_{u+1-R_{j_1}(y)}(\pi F^{j_1}y).
\]
Then let $j=j_1+j_2$ where $j_2\le n/\inf R$ is the lap
number corresponding to $Z_{u+1-R_{j_1}(y)}(F^{j_1}y)$ under
$Z_n$.  Altogether, we obtain
\[
Z_n\pi Z_{u+1}y=Z_{n+u+1-R_{j_1}(y)-R_{j_2}(F^{j_1}y)}(F^{j_2}\pi F^{j_1}y)
= p(F^{j_2}\pi F^{j_1}y,n+u+1-R_j(y)).
\]
Similarly, we write
\[
Z_n\pi Z_{u+1}y' = p(F^{j'_2}\pi F^{j'_1}y',n+u+1-R_{j'}(y')).
\]
Again, we consider the three cases $j=j'$, $j=j'+1$,
$j=j'-1$ separately.  For example, if $j'=j+1$, then
\begin{align*}
 & |Z_n\pi Z_{u+1}y-Z_n\pi Z_{u+1}y'| \\ & \quad =
|p(F^{j_2}\pi F^{j_1}y,n+u+1-R_j(y))-
p(F^{j'_2}\pi F^{j'_1}y',n+u+1-R_{j'}(y'))| \\ & \quad \le 
|p(F^{j_2}\pi F^{j_1}y,n+u+1-R_j(y))- 
p(F^{j_2}\pi F^{j_1}y,R(F^jy)| \\ & \qquad\qquad
+ |p(F^{j_2+1}\pi F^{j_1}y,0)-
p(F^{j'_2}\pi F^{j'_1}y',n+u+1-R_{j+1}(y'))| \\
& \quad \le C\{R_{j+1}(y)-n-u-1\}+C\{\theta^{s(F^{j_2+1}\pi F^{j_1}y,F^{j'_2}\pi F^{j'_1}y')}+n+u+1-R_{j+1}(y')\} \\
& \quad = C\{\theta^{s(F^{j+1}y,F^{j+1}y')}+R_{j+1}(y)-R_{j+1}(y')\},
\end{align*}
and the calculation proceeds as for~\eqref{eq-A1}.  Hence we obtain
\begin{align} \label{eq-A2}
|A(\pi Z_{u+1}y,\pi Z_{u+1}y')| \le C|Dv|_\infty(1+|R|_\theta)\theta^N.    
\end{align}

Combining~\eqref{eq-initial},~\eqref{eq-B},~\eqref{eq-A1},~\eqref{eq-A2},
we obtain that $|\bar v(y,u)-\bar v(y',u)|\le C|Dv|_\infty\{
e^{-aN}+(1+|R|_\theta)\theta^{N/q}\}$.
Hence~\eqref{eq-barv_y} holds with
$\theta'=\max\{e^{-a},\theta^{1/q}\}$ completing the proof.
\end{proof}


\def\cprime{$'$}


\end{document}